\documentclass[11pt, twoside]{article}

\usepackage{amsmath,amsthm}

\usepackage[utf8]{inputenc}

\usepackage{enumerate}

\usepackage{fancyhdr}
\usepackage{cite}
\usepackage{enumitem}

\usepackage{graphics}

\usepackage{algcompatible}
\usepackage{algorithm}
\usepackage{algpseudocode}

\usepackage{xargs}                      
\usepackage[colorinlistoftodos,prependcaption,textsize=tiny]{todonotes}
\usepackage[normalem]{ulem}

\newif\ifpreprint

\usepackage{caption}
\usepackage{subcaption}

\usepackage{hyperref}
\hypersetup{%
    colorlinks=True, 
    citecolor=blue,
    linkcolor=black}

\usepackage[charter]{mathdesign}

\newcommand{\rom}[1]{\uppercase\expandafter{\romannumeral #1\relax}}

\newcommand{\black}{\color{black}}

\usepackage{environ}
\NewEnviron{cases2}[1]{%
                      \begin{equation}\label{#1}\left\{ \begin{alignedat}{2} \BODY \end{alignedat}\right. \end{equation} }
\NewEnviron{cases3}[1]{%
                      \begin{equation}\label{#1} \begin{alignedat}{2} \BODY \end{alignedat} \end{equation} }
\NewEnviron{cases22}{%
                      \begin{equation}\left\{ \begin{alignedat}{2} \BODY \end{alignedat}\right. \end{equation} }

\theoremstyle{plain}
  \newtheorem{theorem}{Theorem}[section]
  \newtheorem{corollary}[theorem]{Corollary}
  
  \newtheorem{lemma}[theorem]{Lemma}
  \newtheorem{definition}[theorem]{Definition}
  \newtheorem{remark}[theorem]{Remark}
   
  \newtheorem{example}[theorem]{Example}

  \newtheorem*{assumption*}{Assumption}

  \newcommand{\vk}[1]{v^{(#1)}}
  
  \newcommand{\Uk}[1]{U^{(#1)}}
  \newcommand{\Zk}[1]{Z^{(#1)}}

  \newcommand{\frR}[1]{\Cr^{(#1)}}
  \newcommand{\frT}[1]{\Ct^{(#1)}}

  \newcommand{\Rki}[2]{R^{(#1)}_{#2}}

  \numberwithin{equation}{section}
    
\providecommand{\Div}{\operatorname{div}}          

\newcommand{\VJ}{{\mathbf{J}}}

\newcommand{\VN}{{\mathbf{N}}}

\newcommand{\VR}{{\mathbf{R}}}

\newcommand{\Dsf}{{\mathsf{D}}}

\providecommand{\Ca}{{\cal A}}

\providecommand{\Ce}{{\cal E}}

\providecommand{\Cj}{{\cal J}}

\providecommand{\Cp}{{\cal P}}

\providecommand{\Cr}{{\cal R}}
\providecommand{\Cs}{{\cal S}}
\providecommand{\Ct}{{\cal T}}

\setenumerate[0]{label=(\alph*)}

\newcommand{\eps}{{\varepsilon}}

\newcommand{\ben}{\begin{equation}}
\newcommand{\een}{\end{equation}}

\newcommand{\benn}{\begin{equation*}}
\newcommand{\eenn}{\end{equation*}}

\usepackage[colorinlistoftodos]{todonotes}

\usepackage{authblk}
\usepackage[T1]{fontenc}
\usepackage[utf8]{inputenc}

\title{Second order Taylor-like topological expansion  with multiple holes}
\author{Kevin Sturm}

\begin{document}
\maketitle
\begin{abstract}
In this paper we derive a framework for topological expansions of shape functionals with respect to multiple holes based on the expansion of one hole. Our strategy is similar to deriving a Taylor expansion of order two of functions in $\VR^d$ by applying $d$-times a Taylor expansion in $\VR$.

For shape functionals defined on subsets of $\VR^2$ we relate the topological expansion with multiple holes to a Taylor expansion with respect to the ball volumes where usual
derivatives are replaced by topological state derivatives. The first topological state derivative was introduced in a previous paper and relates the first topological derivative of cost functionals. The second topological state derivative at two distinct points is introduced, similar to a second directional derivatives of functions defined on Banach spaces, 
as the topological state derivative of the topological state derivative. The second topological state derivative at the same point is defined as the "far-away" component of the compound layer expansion of the state variable. 

In dimension three the obtained topological expansion still involves first and second topological state derivatives, but it does not follow the same pattern of a Taylor expansion in the ball volumes. As we will show this dimension dependent behaviour is rather of geometrical nature and already appears in cost functionals without PDE constraints. 
\end{abstract}

\tableofcontents

\section{Introduction}

In this paper, we study higher order topological expansions with multiple topological perturbations simultaneously of cost functional $J(\Omega) = \Cj(u_\Omega)$ with $u_\Omega:\Dsf\to \VR$ being the solution to 
\begin{cases2}{eq:intro1}
-\Delta u_\Omega + \gamma u_\Omega & = f_{\Omega} && \quad \text{ in } \Dsf\\
    u_\Omega & = g && \quad \text{ on } \Gamma \\
    \partial_n u_\Omega & = h && \quad \text{ on } \Gamma_N,
\end{cases2}
with $\gamma\ge 0$ a non-negative number, $f_\Omega$ a piecewise constant function and $g,h$ given smooth enough data. 

Topological derivatives have been thoroughly studied for single "hole" perturbations. We refer to the pioneering works \cite{a_GAGUMA_2001a,a_SOZO_1999a} and also to the monographs \cite{b_NOSOZO_2019a,b_NOSO_2013a}, the chapter \cite{a_ALDAJO_2021a} and the introductory paper \cite{a_AM_2021a} for examples and  theoretical/numerical results. 

Topological derivatives or even higher order expansions can be computed using different techniques. There are various Lagrangian techniques, e.g., approaches using a perturbed adjoint equation; see Amstutz' approach \cite{a_AM_2006a} or \cite{a_ST_2020a} for the averaged adjoint approach. Another important method to compute higher order topological derivatives is the so-called compound layer expansion \cite{b_MANAPL_2012a, b_MANAPL_2012b} which allows the expansion first of the state variable and then the cost functional. It introduces successive approximations followed by so-called correctors.  As shown in \cite{a_BAGAST_2021b,a_BAST_2021a} the adjoint approaches can be efficiently combined with the compound layer expansion to compute higher order topological derivatives. Other papers dealing with higher order topological expansions \cite{a_BOCO_2017a} examining three dimensional anisotropic elasticity and \cite{a_BO_2009a} studying a  2D potential problems.

Multiple simultaneous topological perturbations of certain cost functionals and linear PDEs have been studied in some publications. Here we quote \cite{a_CALANO_2015a} where the Kohn-Volgelius cost functional for the source reconstruction for the Laplacian is stuidied and an expansion of multiple holes is proved. In this referenced work the PDE \eqref{eq:intro1} with $\gamma=0$ has been studied. Although we do not cover this special cost functional our analysis can be used to recover this case as well. Other publications such as  EIT \cite{a_HILANO_2011a} also derive asymptotics of cost functionals with respect to multiple simultaneous topological perturbations at once. Notably in \cite{a_HILANO_2011a} the topological perturbation appears in the differential operator and is therefore much more singular than our problem \eqref{eq:intro1}.

In this paper we aim to give a unified and concise approach to the computation of topological expansions with respect to multiple simultaneous topological perturbations. We give a general theorem on how such an expansion can be obtain, which makes it possible to adapt for other problems and cost functionals as well. 
Moreover, we link the second order topological expansion of multiple holes with second order topological state derivatives which obey typical calculus rules. For this we extend the notion of the first topological state derivative introduced in \cite{a_BAMAST_2023a} to the second order. Our analysis shows that there is a link between second order topological state derivatives and the compound layer expansion. Let us mention that for our type of problem (and other low order perturbations with respect to one inclusion/hole) it has been observed already earlier that the first derivative can be computed using classical derivatives in appropriate $L^p$ spaces; see, e.g., Section 2 in \cite{a_AM_2021a}. This also relates to the first order topological derivative of \cite{a_BAMAST_2023a}.

\paragraph{Our key contributions}

\begin{itemize}
    \item systematic derivation technique for topological expansions with multiple simultaneous "holes"
    \item recovering second order topological derivatives with usual calculus rules
    \item introduction of second order topological state derivatives
    \item the role of the compound layer expansion and second order topological state derivatives
\end{itemize}

\subsection{Overview of the results} 
\subsubsection{Taylor-like expansion with multiple inclusions}
Let $\Omega$ be a domain and let $\Omega_{\eps_1}$ be a topological perturbation of it by a ball $B_{\eps_1}:= B_{\eps_1}(x_1)$ of radius $\eps_i>0$ centered at $x_1$. Then we show that we have a topological expansion of the form
\begin{equation}\label{eq:expansion_J_basis_one_hole}
    J(\Omega_{\eps_1}) = J(\Omega) + |B_{\eps_1}|DJ(\Omega)(x_1) + \ell(\eps_1) D^{\frac32}J(\Omega)(x_1,x_1) + \frac{|B_{\eps_1}|^2}{2} D^2J(\Omega)(x_1,x_1) + o(|B_{\eps_1}|^2) 
\end{equation}
with a function $\ell(\eps_1)$ satisfying for small $\eps_1>0$ that $|B_{\eps_1}|>\ell(\eps_1) > |B_{\eps_1}|^2$. The term $DJ(\Omega)(x_1,x_1)$ is the first order topological derivative while $D^{\frac32}J(\Omega)(x_1,x_1)$ and $D^2J(\Omega)(x_1,x_1)$ are defined by the expansion. 

Let now $\Omega_{\eps_1,\ldots,\eps_n}$ be denote $n$-simultaneous topological perturbations of $\Omega$ with balls $B_{\eps_i}:= B_{\eps_i}(x_1)$ of radius $\eps_i>0$ centered at $x_i$. Then we prove by repetitively applying \eqref{eq:expansion_J_basis_one_hole} that we indeed have an expansion of the form
\begin{align}\label{eq:expansion_J_basis_multi_hole}
    \begin{split}
    J(\Omega_{\eps_1,\ldots,\eps_n})  =  J(\Omega) &+ D\VJ (\Omega)(x_1,\ldots, x_n)\cdot \begin{pmatrix}
|B_{\eps_1}|\\ \vdots \\ |B_{\eps_n}|
\end{pmatrix} 
\\
                                                   & +  \text{diag}(D^{\frac32}J(\Omega)(x_1,x_1),\ldots D^{\frac32}J(\Omega)(x_n,x_n)) \begin{pmatrix} \sqrt{\ell(\eps_1)} \\ \ldots \\\sqrt{\ell(\eps_n)}\end{pmatrix}\cdot \begin{pmatrix} \sqrt{\ell(\eps_1)} \\ \ldots \\\sqrt{\ell(\eps_n)}\end{pmatrix}  \\
  & + \frac12D^2 \VJ(\Omega)(x_1,\ldots, x_n) \begin{pmatrix}
|B_{\eps_1}| \\\vdots \\ |B_{\eps_n}|
\end{pmatrix}\cdot \begin{pmatrix}
|B_{\eps_1}| \\ \vdots \\|B_{\eps_n}|
\end{pmatrix}  
+ o(|B_{\eps_1}|^2+\cdots + |B_{\eps_n}|^2),
    \end{split}
\end{align}
where $D\VJ (\Omega)(x_1,\ldots, x_n) \in \VR^n$ is a vector, $D^2 \VJ(\Omega)(x_1,\ldots, x_n) \in \VR^{n\times n}$ is a matrix, which we also refer to the \emph{topological Hessian}.  The term $\ell(\eps)$  satisfies for $\eps>0$ small 
 that $|B_{\eps}| > \ell(\eps) > |B_{\eps}|^2$. Moreover we have the following key observations:
 \begin{itemize}
     \item  For our problem we show for dimension $d=2$ we have $D^{\frac32}J(\Omega)(x_i,x_i)=0$ and can be set $\ell(\eps)=0$.
     \item For dimension $d=3$ we have $\ell(\eps)= \eps^{-1}|B_\eps|^2$ and $D^{\frac32}J(\Omega)(x_i,x_i)\ne 0$.
\item The diagonal elements of $D^2 \VJ(\Omega)(x_1,\ldots, x_n)$ are precisely $D^2J(\Omega)(x_i,x_i)$ from the expansion \eqref{eq:expansion_J_basis_one_hole}. 
\item The off-diagonal element $(D^2 \VJ(\Omega)(x_1,\ldots, x_n))_{ij}$, $i\ne j$, are the first order topological derivative of $\Omega\mapsto DJ(\Omega)(x_i)$ at the points $x_j$ and thus denoted $D^2J(\Omega)(x_i,x_j)$. 
\item In general we have
    \begin{equation}
        \lim_{x_i\to x_j} D^2J(\Omega)(x_i,x_j) \ne D^2J(\Omega)(x_i,x_i).
    \end{equation}
 \end{itemize}

 \subsubsection{Second order expansion as a Taylor expansion with classical calculus rules for $d=2$}
The second main result of our paper is that the first derivative vector $D\VJ (\Omega)(x_1,\ldots, x_n)$ and the matrix $D^2 \VJ(\Omega)(x_1,\ldots, x_n)$ in the expansion \eqref{eq:expansion_J_basis_multi_hole} can be computed like usual derivatives using so-called state derivatives. We consider the following cost functional for $d=2$:
\begin{equation}
    J(\Omega) = \int_\Dsf(u_\Omega -u_r)^2\;dx
\end{equation}
with some target function $u_r$, where $u_\Omega$ is the solution to \eqref{eq:intro1}. Then we show that there are operators $d_{x_i}$ and $d^2_{x_ix_j}$ ($d_{x_i}u_\Omega$ and $d_{x_ix_j}^2u_\Omega$ are the topological state derivatives), such that
\begin{equation}
    DJ(\Omega)(x_i) = d_{x_i} \int_\Dsf(u_\Omega -u_r)^2\;dx = \int_\Dsf2(u_\Omega -u_r) d_{x_i} u_\Omega\;dx
\end{equation}
and also 
\begin{align}
    \begin{split}
        D^2J(\Omega)(x_i,x_j) & = d_{x_j} d_{x_i} \int_\Dsf(u_\Omega -u_r)^2\;dx \\
                              & = d_{x_j}\int_\Dsf2(u_\Omega -u_r) 2 d_{x_i} u_\Omega\;dx \\
                              & = \int_\Dsf 2 (d_{x_i} u_\Omega)(d_{x_j}u_\Omega) + 2(u_\Omega-u_r) d_{x_ix_j}^2u_\Omega\;dx.
\end{split}
\end{align}
Moreover, we have a Taylor-like expansion in the volume $|B_{\eps_i}|:=|B_{\eps_i}(x_i)|$:
\begin{equation}
    J(\Omega_{\eps_i})  = J(\Omega) + |B_{\eps_i}| DJ(\Omega)(x_i) + \frac{|B_{\eps_i}|^2}{2} D^2J(\Omega)(x_i,x_i) + o(|B_{\eps_i}|^2). 
\end{equation}

\begin{itemize}
    \item The operator $d_{x_i}$ obeys the usual calculus rules, such as, product and chain rule.
    \item The second order operator is symmetric $d_{x_ix_j} = d_{x_jx_i}$.
    \item The second order state derivative $d_{x_ix_i}^2u_\Omega $ at $x_i$ are closely related to the second order term of the compound layer expansions of $u_\Omega$ with respect to the inclusion $B_{\eps_i}(x_i)$
    \item The state derivative $d_{x_i}u_\Omega$ corresponds to the volume $|B_{\eps_i}(x_i)|$.
    \item The state derivative $d_{x_ix_j}^2u_\Omega$ corresponds to the volume $|B_{\eps_i}(x_i)||B_{\eps_j}(x_j)|$. 
\end{itemize}

 \paragraph{Notation}
 Henceforth we let $\Dsf\subset \VR^d$, $d\in\{2,3\}$ be a bounded $C^1$ domain and $\Omega \subset \Dsf$ an open subset. We divide $\partial \Dsf$ into two smooth boundaries $\Gamma,\Gamma_N\subset \partial\Dsf$ that are disjoint and $\bar \Gamma_N \cup \bar\Gamma =\partial \Dsf$. We denote by $W^{k,p}(\Dsf)$ the usual Sobolev space  with differentiability $k\ge 1$ and integrability $p\in [1,\infty]$. The space $L^p(\Dsf)$ denotes the measurable function on $\Dsf$ that are $p$-integrable. The space $W^{1,p}_0(\Dsf)$ denotes the subspace of $W^{1,p}(\Dsf)$ with trace zero on $\partial \Dsf$. Similarly $W^{1,p}_\Gamma(\Dsf)$ denotes the subspace of $W^{1,p}(\Dsf)$ with trace zero on $\Gamma$. The space $\VR^d$ is usually endowed with the Euclidean norm $|x|^2:= \sum_{i=1}^d |x_i|^2$, $x=(x_1,\ldots, x_d)\in \VR^d$. The Euclidean ball at a point $x\in \VR^d$ with radius $\eps$ is denoted by $B_\eps(x)$. If $x=0$ we use the notation $B_\eps:= B_\eps(0)$ and if additionally $\eps=1$, then we set $B:= B_1(0)$.

\section{Topological derivatives and multiple holes}
Henceforth we denote by $\Cp(\Dsf)$ the powerset of a set $\Dsf\subset \VR^d$. The domain $\Dsf$ is referred to as hold-all domain where all shapes are contained. A function $J:\Ca\subset \Cp(\Dsf)\to \VR$ is called a shape functional.  

The ball of radius $\eps>0$ centered $x_1\in \VR^d$ is denoted by $B_\eps(x_1)$. For the unit ball at the center $x_1=0$ with radius $\eps>0$, we set $B_\eps:= B_\eps(0)$ and for the unit ball at the center we simply write $B$. 
\subsection{First and second topological derivative}

\begin{definition}\label{def:first_and_second}
 Let $\Ca$ be an admissible set containing all open subsets of $\Dsf$.   Let $J:\Ca\to \VR$ be a shape functional and an open set $\Omega\subset \Dsf$ be given.
    \begin{itemize}
        \item We define the first order topological derivative at $x_1\in \Dsf\setminus \partial \Omega$:
\begin{equation}
    DJ(\Omega)(x_1):=  \lim_{\eps\searrow0 }\begin{cases}
        \frac{J(\Omega\setminus\overline{B_\eps(x_1)}) - J(\Omega)}{|B_\eps(x_1)|} & \text{ for } x_1 \in \Omega \\
        \frac{J(\Omega\cup  B_\eps(x_1)) - J(\Omega)}{| B_\eps(x_1)|} & \text{ for } x_1 \in \Dsf\setminus\overline{\Omega}.
    \end{cases}
\end{equation}
\item The second order topological derivative at $(x_1,x_2)\in \Dsf\setminus \partial \Omega$ for $x_1\ne x_2$ is defined by 
\begin{equation}
    D^2J(\Omega)(x_1,x_2):=  \lim_{\eps\searrow0 }\begin{cases}
        \frac{DJ(\Omega\setminus\overline{B_\eps(x_2)})(x_1) - DJ(\Omega)(x_1)}{| B_\eps(x_2)|} & \text{ for } x_2 \in \Omega \\
        \frac{DJ(\Omega\cup  B_\eps(x_2))(x_1) - DJ(\Omega)(x_1)}{| B_\eps(x_1)|} & \text{ for } x_2 \in \Dsf\setminus\overline{\Omega}.
    \end{cases}
\end{equation}
\end{itemize}

\end{definition}

\begin{definition}
For an open set $\Omega\subset \Dsf$ and $x_1\in \Dsf\setminus \partial \Omega$ we introduce the notation 
\begin{equation}
    \Omega_\eps(x_1) := \begin{cases}
        \Omega\setminus B_\eps(x_1) & \text{ for } x_1\in \Dsf\setminus\overline \Omega \\
        \Omega\cup B_\eps(x_1) & \text{ for } x_1\in \Omega. 
    \end{cases}
\end{equation}
With this definition we can write the first and second topological derivative in the compact form: 
\begin{equation}
    DJ(\Omega)(x_1) = \lim_{\eps \searrow 0} \frac{J(\Omega_\eps(x_1))-J(\Omega)}{|B_\eps(x_1)|} 
\end{equation}
and assuming $x_1\ne x_2$:
\begin{equation}
     D^2J(\Omega)(x_1,x_2) = \lim_{\eps \searrow 0} \frac{DJ(\Omega_\eps(x))(x_1)-DJ(\Omega)(x_1)}{|B_\eps(x_2)|}.
\end{equation}
\end{definition}

Note that in Definition~\ref{def:first_and_second} the "diagonal" case $D^2J(\Omega)(x_1,x_2)$ for $x_1=x_2$ is not defined yet. This definition is outlined next. 
\begin{definition}\label{def:DJ2_diagonal}
    Assume that for $x_1\in \Dsf\setminus \partial \Omega$ there are $\Cs_\Omega(x_1)\in \VR$ and a mapping $\eps\mapsto \ell_{\Omega,x_1}(\eps)$ on $ (0,\zeta]\to \VR, \zeta>0$, such that for $\eps>0$ small
\begin{equation}
    J(\Omega_\eps(x_1)) = J(\Omega) + |B_\eps(x_1)| DJ(\Omega)(x_1) + \ell_{\Omega,x_1}(\eps) + \frac{|B_\eps(x_1)|^2}{2}\Cs_\Omega(x_1) + o(|B_\eps(x_1)|^2)
\end{equation}
with $|B_\eps|^2<\ell_{\Omega,x_1}(\eps)<|B_\eps|$ for all small $\eps>0$. Then we define:
\begin{equation}
    D^2J(\Omega)(x_1,x_1):= \Cs_\Omega(x_1).
\end{equation}
So the "diagonal" element $D^2J(\Omega)(x_1,x_1)$ corresponding to the term $\frac{|B_\eps(x_1)|^2}{2}$ in the asymptotic expansion in the "hole" $x_1$. 
\end{definition}

\begin{remark}
In the next example we will later see that for $d=3$, we have
\begin{equation}
    \lim_{y\to x_1} D^2J(\Omega)(x_1,y) = D^2J(\Omega)(x_1,x_1).
\end{equation}
However, from our study in this paper it seems that this is only in specific cases true. 
\end{remark}

\begin{definition}
For every set $\Omega\subset \Dsf$ we define the indicator function as:
\begin{equation}
    \sigma_\Omega(x):= \begin{cases}
 -1& \text{ for } x\in \Omega, \\
 1& \text{ for } x\in \Omega^c.
    \end{cases}
\end{equation}
\end{definition}

Let us finish this section with a simple example.

\begin{example}
Let us consider the simple example
\begin{equation}
    J(\Omega)= \int_\Omega f\;dx, \quad f\in C^{d+1}(\VR^d)
\end{equation}
for a bounded domain $\Omega\subset \VR^d$.
\begin{equation}
    DJ(\Omega)(x_1) = \sigma_\Omega(x_1), \quad D^2J(\Omega)(x_1,x_1)=\sigma_\Omega(x_1) \frac{2}{|B|^2 d!}\int_B \nabla^d f(x_1)[x]^d\;dx.
\end{equation}
 \end{example}
\begin{proof}
 We have
\begin{equation}
    J(\Omega_\eps(x_1))- J(\Omega) = \sigma_\Omega(x_1)\int_{B_\eps(x_1)} f(x)\;dx = \sigma_\Omega(x_1) \eps^d \int_{B} f(x_1 + \eps x)\;dx.
\end{equation}
  So a Taylor expansion on the right hand side yields
  \begin{equation}\label{eq:DJ_f_full}
      J(\Omega_\eps(x_1))- J(\Omega) = \sigma_\Omega(x_1)\eps^d\left(|B|f(x_1) + \eps \int_{B} \nabla f(x_1)\cdot x\;dx + \cdots + \eps^d\frac{1}{d!}\int_{B} \nabla^d f(x_1)[x]^d\;dx)\right) + o(\eps^{2d}).
\end{equation}
From the definition of the first and second topological derivative the result follows. 
\end{proof}

\begin{remark}\label{rem:int_nabla_f_xk}
    Note that for $k\ge 1$ odd we have that $\int_B \nabla^k f(x_1)[x]^k\;dx = 0$.
    \begin{proof}
        Since  with $B^+:= \{x\in B:\; x_1 >0\}$ and $B^-:= \{x\in B:\; x_1 < 0\}$, we can write  
    \begin{equation}\label{eq:int_nabla_k}
        \int_B \nabla^k f(x_1)[x]^k\;dx = \int_{B^+} \nabla^k f(x_1)[x]^k\;dx + \int_{B^-} \nabla^k f(x_1)[x]^k\;dx.
    \end{equation}
    Now $-B^- = B^+$ and thus changing variables $T(x):= -x$ we obtain 
    \begin{align}
        \begin{split}
        \int_{B^+} \nabla^k f(x_1)[x]^k\;dx & = \int_{T(B^-)}\nabla^k f(x_1)[x]^k\;dx \\
                                            & = \int_{B^-}\nabla^k f(x_1)[-x]^k\;dx  \\
                                            &= (-1)^k\int_{B^-} \nabla^k f(x_1)[x]^k\;dx.
    \end{split}
\end{align}
Now $(-1)^k=-1$ since $k$ is odd and it follows that  the integral on the left hand side in \eqref{eq:int_nabla_k} vanishes.
\end{proof}
\end{remark}

\begin{lemma}\label{lem:example_int_f}
Let $f$ and $J$ be as in the previous lemma. Let $\Omega\subset \Dsf$ be a bounded domain and $x_1\in \Dsf\setminus\partial \Omega$. 
\begin{itemize}
    \item[(a)] For dimension $d=2$, we have
\begin{equation}\label{eq:DJ_f}
    DJ(\Omega)(x_1) = \sigma_\Omega(x_1)f(x_1) \quad \text{ and } \quad D^2J(\Omega)(x_1,x_1)= \sigma_\Omega(x_1) \frac{\Delta f(x_1)}{2\pi}.
\end{equation}
In particular,
\begin{equation}\label{eq:expansion_Jf}
    J(\Omega_\eps(x_1)) = J(\Omega) + \sigma_\Omega(x_1)\left( |B_\eps(x_1)|f(x_1) + \frac{|B_\eps(x_1)|^2}{2} \frac{\Delta f(x_1)}{2\pi}\right) + O(\eps^6). 
\end{equation}
\item[(b)] In dimension $d=3$ we have $D^2J(\Omega)(x_1,x_1)=0$ and 
    \begin{equation}\label{eq:expansion_Jf2}
        J(\Omega_\eps(x_1)) = J(\Omega) + \sigma_\Omega(x_1)\left( |B_\eps(x_1)|f(x_1) + \frac{\eps^{-1}|B_\eps(x_1)|^2}{2} \frac{3\Delta f(x_1)}{20\pi} \right)+ O(\eps^7).
    \end{equation}
\end{itemize}
\end{lemma}
\begin{proof}
    \emph{ad (a):} First note that $\frac12 (\nabla (|x|^2-1)) = x$. Therefore partial integration yields
\begin{align}
    \begin{split}
    D^2J(\Omega)(x_1,x_1) & = \sigma_\Omega(x_1) \frac{1}{|B|^2} \frac12 \int_{B} \nabla^2 f(x_1)x\cdot \nabla (|x|^2-1)\;dx \\
                      & = -\sigma_\Omega(x_1) \frac{1}{|B|^2} \frac12 \int_{B} (|x|^2-1) \underbrace{\text{div}(\nabla^2 f(x_1) x)}_{= \Delta f(x_1)} \;dx.
\end{split}
\end{align}
Finally noting $|B| = \pi$ and $\int_{B} |x|^2-1\;dx = -\pi/2$ the equation \eqref{eq:DJ_f} follows. 

\emph{ad (b):} The expansion \eqref{eq:expansion_Jf2} follows from the Taylor expansion \eqref{eq:DJ_f_full} and a partial integration as in point (a) and $\int_B |x|^2-1\;dx = -\frac{8\pi}{15}$.
\end{proof}

\subsection{Abstract expansion of cost functionals with respect to multiple holes}
For $x_1,x_2\in \Dsf\setminus \partial \Omega$ with $x_1\ne x_2$, we define $B_{\eps_i}:= B_{\eps_i}(x_i)$, $i=1,2$ and 
\begin{equation}
    \Omega_{\eps_1}:= \Omega_{\eps_1}(x_1), \quad \Omega_{\eps_1,\eps_2}:= \Omega_{\eps_1,\eps_2}(x_1,x_2):= \begin{cases}
     \Omega_{\eps_1}\setminus B_{\eps_2}(x_2) & x_2\in \Omega \\
     \Omega_{\eps_1} \cup B_{\eps_2}(x_2) & x_2\in \Dsf\setminus\overline{\Omega}.
    \end{cases}
\end{equation}
For $x_1,\ldots, x_n$, $n\ge3$, with $x_i\ne x_j$ for $i\ne j$ the set $\Omega_{\eps_1,\ldots,\eps_n}(x_1,\ldots,x_n)=\Omega_{\eps_1,\ldots,x_n}$ is defined in the same fashion. 

Now we prove a theorem which allows to obtain topological expansion of $n$-holes by reducing it to iteratively applying the "one hole" expansion.  We set $B_\eps := B_\eps(0)$. Then we note that $|B_\eps|=|B_\eps(x_1)|$ for all $x_1\in \Dsf\setminus\partial \Omega$. 

\begin{definition}
Let $(\Omega,\eps)\mapsto \Cr_\Omega(\eps)$  be a function defined for $\delta >0$ on
\begin{equation}
    (0,\delta]\times \{\Omega\subset \Dsf: \Omega \text{ is an open set} \}. 
\end{equation}
    Let $f:(0,\delta)\to \VR^+$ be a given function. We say $|\Cr_\Omega(\eps)|=O(f(\eps))$ uniformly in $\Omega$, if for all $n\ge 1$,  $x_1,\ldots, x_n\in \Dsf\setminus\partial \Omega$ with $x_i\ne x_j$ for $i\ne j$ there is a constant $\zeta>0$ and $C>0$, such that 
\begin{equation}
    |\Cr_{\Omega_{\eps_1,\ldots, \eps_n}}(\eps)|\le C|f(\eps)| \quad \text{ for all } \eps, \eps_1,\ldots,\eps_n\in (0,\zeta].
\end{equation}
In the same way $|\Cr_\Omega(\eps)|=o(f(\eps))$ uniformly in $\Omega$ is defined. Examples for $f$ are $f(\eps)=|B_\eps|$ or  $f(\eps)=\eps^{-1}|B_\eps|^2$ for $\eps>0$.  
\end{definition}

Given $\delta >0$ let $(0,\delta)\to \VR^+: \eps \mapsto \ell(\eps)$ be a function given function such that $|B_\eps|^2< \ell(\eps) < |B_\eps|$ for all small $\eps>0$ and $\ell(\eps) = o(|B_\eps|)$. Note that $\lim_{\eps\searrow 0}\ell(\eps) =0$.

\begin{theorem}\label{main:thm_expansion_general}
    Let the following assumptions be satisfied. 
\begin{enumerate}
    \item[(A1)] For every open set $\Omega\subset \Dsf$ and every $x_1\in \Dsf\setminus \partial \Omega$, we assume that
    \begin{align}\label{eq:first}
        \begin{split}
            J(\Omega_{\eps_1}) = J(\Omega) + |B_{\eps_1}| DJ(\Omega)(x_1) & + \ell(\eps)D^{\frac32}J(\Omega)(x_1,x_1)  \\
                                                                          & + \frac{|B_{\eps_1}|^2}{2}  DJ^2(\Omega)(x_1,x_1) + \frR{1}_{\Omega,x_1}(\eps_1),
    \end{split}
\end{align}
with  $|\Cr_{\Omega,x_1}(\eps)| = o(|B_\eps|^2)$ uniformly in $\Omega$. 
\item[(A2)]
    For every open set $\Omega\subset \Dsf$ and every $x_1,x_2\in \Dsf\setminus \partial \Omega$, $x_1\ne x_2$, we assume that 
 \begin{equation}
     DJ(\Omega_{\eps_2})(x_1) = DJ(\Omega)(x_1) +|B_{\eps_2}|D^2J(\Omega)(x_1,x_2) + \frR{2}_{\Omega,x_1,x_2}(\eps_2)
\end{equation}
with $|\frR{2}_{\Omega,x_1,x_2}(\eps)| = o(|B_\eps|)$ uniformly in $\Omega$. 
\item[(A3)] For every open set $\Omega\subset \Dsf$ and every $x_1,x_2\in \Dsf\setminus \partial \Omega$, $x_1\ne x_2$, we assume that
\begin{align}
    D^{\frac32}J(\Omega_{\eps_2})(x_1,x_1) & = D^{\frac32}J(\Omega)(x_1,x_1) + \frR{3}_{\Omega,x_1,x_2}(\eps_2)
\end{align}
with $|\frR{3}_{\Omega,x_1,x_2}(\eps)| = O(|B_\eps|)$ uniformly in $\Omega$. 
\item[(A4)] For every open set $\Omega\subset \Dsf$ and every $x_1,x_2,x_3\in \Dsf\setminus \partial \Omega$, $x_1\ne x_3$ and $x_2\ne x_3$, assume that
\begin{align}
    \begin{split}
        D^2 J(\Omega_{\eps_3})(x_1,x_2) & = D^2 J(\Omega)(x_1,x_2) + \frR{4}_{\Omega,x_1,x_2}(\eps_3)
\end{split}
\end{align}
with  $|\frR{4}_{\Omega,x_1,x_2}(\eps_2)| = o(1)$  uniformly in $\Omega$. 
\end{enumerate}
Then we have for $n\ge 1$ and  $x_1,\ldots, x_n\in \Dsf\setminus\partial \Omega$  with $x_i\ne x_j$ for $i\ne j$:
\begin{align}\label{eq:expansion_proof_e1_e2_}
    \begin{split}
    J(\Omega_{\eps_1,\ldots,\eps_n})  =  J(\Omega) +& D\VJ (\Omega)(x_1,\ldots, x_n)\cdot \begin{pmatrix}
|B_{\eps_1}|\\ \vdots \\ |B_{\eps_n}|
\end{pmatrix} 
+  D^{\frac32} \VJ(\Omega)(x_1,\ldots, x_n) \begin{pmatrix}
    \sqrt{\ell(\eps_1)} \\\vdots \\ \sqrt{\ell(\eps_n)}
\end{pmatrix}\cdot \begin{pmatrix}
\sqrt{\ell(\eps_1)} \\ \vdots \\ \sqrt{\ell(\eps_n)}
\end{pmatrix}  
                              \\
& +\frac12D^2 \VJ(\Omega)(x_1,\ldots, x_n) \begin{pmatrix}
|B_{\eps_1}| \\\vdots \\ |B_{\eps_n}|
\end{pmatrix}\cdot \begin{pmatrix}
|B_{\eps_1}| \\ \vdots \\|B_{\eps_n}|
\end{pmatrix}  
          + o(|B_{\eps_1}|^2+\cdots + |B_{\eps_n}|^2),
    \end{split}
\end{align}
where
\begin{align}
    D^2 \VJ(\Omega)(x_1,\ldots, x_n) & :=  \begin{pmatrix}
        D^2J(\Omega)(x_1,x_1) &\cdots &  D^2J(\Omega)(x_1,x_n) \\
                 \vdots             &\ddots & \vdots \\
        D^2J(\Omega)(x_1,x_n) &\cdots  & D^2J(\Omega)(x_n,x_n)
\end{pmatrix}\\
        D^{\frac32} \VJ(\Omega)(x_1,\ldots, x_n) & =  \begin{pmatrix}
            D^{\frac32}J(\Omega)(x_1,x_1) &0 &  0 \\
                 0             &\ddots & \vdots \\
                0  &\cdots  & D^{\frac32}J(\Omega)(x_n,x_n)
\end{pmatrix}\\
    D\VJ (\Omega)(x_1,\ldots, x_n) &:=   \begin{pmatrix}
DJ(\Omega)(x_1) \\ \vdots \\ DJ(\Omega)(x_n)
\end{pmatrix} 
    \end{align}
\end{theorem}
\begin{proof}
    We only prove the case for two "holes". The general case then follows by induction. By Assumption (A1) applied to $\Omega_{\eps_1}$ and $\Omega_{\eps_1,\eps_2}$, we have
\begin{align}\label{eq:expansion_general1}
    \begin{split}
        J(\Omega_{\eps_2}) = J(\Omega) + |B_{\eps_2}| DJ(\Omega)(x_2) &+ \ell(\eps_2) D^{\frac32}J(\Omega)(x_2,x_2) \\
                                                                      & +  |B_{\eps_2}|^2 \frac12 D^2J(\Omega)(x_2,x_2) + \frR{1}_{\Omega,x_2}(\eps_2) 
\end{split}
\end{align}
and
\begin{align}\label{eq:expansion_general2}
    \begin{split}
        J(\Omega_{\eps_1,\eps_2}) = J(\Omega_{\eps_2}) + |B_{\eps_1}|DJ(\Omega_{\eps_2})(x_1) &+ \ell(\eps_1) D^{\frac32}J(\Omega_{\eps_2})(x_1,x_1)\\
                                                                                              & + |B_{\eps_1}|^2 \frac12 D^2J(\Omega_{\eps_2})(x_1,x_1) + \frR{1}_{\Omega_{\eps_2},x_1}(\eps_1).
\end{split}
\end{align}
By Assumptions (A2)-(A4) we have for $x_1\ne x_2$:
\begin{align}\label{eq:expansion_general3}
    \begin{split}
        DJ(\Omega_{\eps_2})(x_1) & = DJ(\Omega)(x_1)  + |B_{\eps_2}| D^2J(\Omega)(x_1,x_2) + \frR{2}_{\Omega,x_1,x_2}(\eps_2)  \\
        D^{\frac32}J(\Omega_{\eps_2})(x_1,x_1) & = D^{\frac32}J(\Omega)(x_1,x_1) + \frR{3}_{\Omega,x_1,x_2}(\eps_2) 
    \end{split}
\end{align}
and 
\begin{align}\label{eq:expansion_general4}
    \begin{split}
    D^2J(\Omega_{\eps_2})(x_1,x_1) & = D^2J(\Omega)(x_1,x_1) + \frR{4}_{\Omega,x_1,x_1}(\eps_2).
\end{split}
\end{align}
Therefore plugging \eqref{eq:expansion_general1},\eqref{eq:expansion_general3}-\eqref{eq:expansion_general4} into \eqref{eq:expansion_general3} yields:
\begin{align}\label{eq:expansion_proof_e1_e2}
    \begin{split}
        J(\Omega_{\eps_1,\eps_2}) = J(\Omega) & + |B_{\eps_1}|DJ(\Omega)(x_1) + |B_{\eps_2}|DJ(\Omega)(x_2) + \ell(\eps_1) D^{\frac32}J(\Omega)(x_1,x_1) + \ell(\eps_2) D^{\frac32}J(\Omega)(x_2,x_2)\\
                               & + \frac12 |B_{\eps_1}|^2 DJ(\Omega)(x_1,x_1) + \frac12 |B_{\eps_2}|^2DJ(\Omega)(x_2,x_2) + |B_{\eps_1}||B_{\eps_2}| DJ(\Omega)(x_1,x_2) \\
                               & + \Cr_\Omega(\eps_1,\eps_2),
\end{split}
\end{align}
where the remainder is given by
\begin{align*}
    \Cr_\Omega(\eps_1,\eps_2) =&  |B_{\eps_1}|\frR{2}_{\Omega,x_1,x_2}(\eps_2) + \ell(\eps_1) \frR{3}_{\Omega,x_1,x_2}(\eps_2) \\
                               & +  \frac12 |B_{\eps_1}|^2\frR{4}_{\Omega,x_1,x_2}(\eps_2)  
                               + \frR{1}_{\Omega,x_2}(\eps_2) + \frR{1}_{\Omega_{\eps_2},x_1}(\eps_1).
\end{align*}
Now note that 
\begin{equation}
    \ell(\eps_1)  |\frR{3}_{\Omega,x_1,x_2}(\eps_2)| = o(|B_{\eps_2}|) O(|B_{\eps_1}|) = o(|B_{\eps_1}|^2 + |B_{\eps_2}|^2) 
\end{equation}
and $|\frR{1}_{\Omega_{\eps_2},x_1}(\eps_1)| = o(|B_{\eps_1}|^2)$ uniformly in $\Omega$. It follows that $|\Cr_\Omega(\eps_1,\eps_2)| = o(|B_{\eps_1}|^2 + |B_{\eps_2}|^2)$. Finally it is readily checked that \eqref{eq:expansion_proof_e1_e2} can be brought into the form \eqref{eq:expansion_proof_e1_e2_}, which finishes the proof.
\end{proof}

Later on we will discuss several concrete examples which fall into the framework of the previous theorem.

\subsection{Auxiliary results: scaling inequalities and remainder estimates}
In this section we recall the scaling behavior of some well-known inequalities under the affine transformation $T_\eps(x) = \eps x + x_1$ for a fixed point $x_1$.  For the proofs we refer to \cite[Section~3]{a_BAST_2021a}.

\begin{definition}
    For $\eps>0$ we define the inflation of a set $A\subset \VR^d$ by $A_\eps^{-1}:= T_\eps^{-1}(A)$, where $T_\eps(x) = x_1 + \eps x$ and $x_1\in A$. 
\end{definition}

\begin{definition}
For  $x_1\in \Dsf$ and $T_\eps(x)=\eps x+ x_1$, we define the scaled $H^1$-norm on $\Dsf_\eps^{-1}$ for $\eps >0$:
\begin{equation}
    \|\varphi\|_\eps:=\eps\|\varphi\|_{L^2(\Dsf_\eps^{-1})}+\|\nabla \varphi\|_{L^2(\Dsf_\eps^{-1})^{d}}, \quad \varphi\in H^1(\Dsf_\eps^{-1}).
\end{equation}
\end{definition}
\black

\begin{lemma}\label{lem:change_variables_eps}
Let $\eps>0$.
\begin{itemize}

    \item[(a)] For $1\le p<\infty$ and $\varphi \in L^p(\Dsf_\eps^{-1})$ there holds
\begin{equation}
    \eps^{\frac{d}{p}}\|\varphi\|_{L^p(\Dsf_\eps^{-1})}=\|\varphi\circ T_\eps^{-1}\|_{L^p(\Dsf)}.
\end{equation}

\item[(b)] For $1\le p<\infty$ and $\varphi \in W^1_p(\Dsf_\eps^{-1})$ there holds
\begin{equation}
    \eps^{\frac{d}{p}-1}\|\nabla \varphi\|_{L^p(\Dsf_\eps^{-1})}=\|\nabla(\varphi\circ T_\eps^{-1})\|_{L^p(\Dsf)}.
\end{equation}

\item[(c)] For $\varphi \in H^1(\Dsf_\eps^{-1})$ there holds
\begin{equation}
\|\varphi\circ T^{-1}_\eps\|_{H^1(\Dsf)} = \eps^{\frac{d}{2}-1}\|\varphi\|_{\eps}.
\end{equation}

\end{itemize}
\end{lemma}

The next lemma rephrases classical inequalities for Sobolev spaces and Lp-spaces on a fixed domain $\Dsf$ in terms of the scaled domain $\Dsf_\eps^{-1}$; cf \cite{a_BAST_2021a}. 
\begin{lemma}[{\hspace{-0.001em}\cite[Lemma~3.4]{a_BAST_2021a}}]\label{lem:scaling_inequality}
    Let $\Dsf \subset \VR^d$ be a bounded Lipschitz domain, $\Gamma \subset \partial \Dsf$ and let $\eps >0$. Recall the definitions $\Dsf_\eps^{-1} =T_\eps^{-1}(\Dsf)$ and $\Gamma_\eps^{-1} = T_\eps^{-1}(\Gamma)$.
\begin{itemize}

    \item[(a)] For $1\le p\le q\le\infty$ there exists a constant $C>0$, such that
\begin{equation}
    \|\varphi\|_{L^p(\Dsf_\eps^{-1})} \le C \eps^{\frac{d}{q}-\frac{d}{p}}\|\varphi\|_{L^q(\Dsf_\eps^{-1})}.
\end{equation}

\item[(b)] Let $d\ge3$ and $2^\ast$ denote the Sobolev conjugate of $2$. There exists a constant $C>0$, such that
\begin{equation}
    \|\varphi\|_{L^{2^\ast}(\Dsf_\eps^{-1})}\le C\|\varphi\|_\eps.
\end{equation}

\item[(c)] Let $d=2$ and $\alpha>0$ small. There exists a constant $C>0$ and $\delta>0$ small, such that
\begin{equation}
    \|\varphi\|_{L^{(2-\delta)^\ast}(\Dsf_\eps^{-1})}\le C \eps^{-\alpha}\|\varphi\|_\eps.
\end{equation}

\item[(d)] For $\varphi\in H^1(\Dsf_\eps^{-1})$ we have
\begin{equation}
    \|\varphi\|_{L^2(\Gamma_\eps^{-1})^d}\le C \eps^{-\frac{1}{2}}\|\varphi\|_\eps.
\end{equation}

\item[(e)] Given a smooth connected domain $\Gamma \subset \partial \Dsf$, there is a continuous extension operator\newline $Z_{\Gamma_\eps^{-1}}:H^{\frac{1}{2}}(\Gamma_\eps^{-1})\rightarrow H^1(\Dsf_\eps^{-1})$, such that
\begin{equation}
    \|Z_{\Gamma_\eps^{-1}}(\varphi)\|_\eps \le C (\eps^{\frac{1}{2}}\|\varphi\|_{L^2(\Gamma_\eps^{-1})^d}+|\varphi|_{H^{\frac{1}{2}}(\Gamma_\eps^{-1})}),\quad \text{ for all }\varphi \in H^{\frac{1}{2}}(\Gamma_\eps^{-1}),
\end{equation}
where $C>0$ is independent of $\eps$.

\item[(f)] Let $\Gamma \subset \partial \Dsf$ have positive measure. There exists a constant $C>0$, such that
\begin{equation}
    \|\varphi\|_{L^2(\Dsf_\eps^{-1})}\le C\eps^{-1}\|\nabla \varphi\|_{L^2(\Dsf_\eps^{-1})^{ d}},\quad \text{ for all }\varphi\in H_{\Gamma_\eps^{-1}}^1(\Dsf_\eps^{-1}).
\end{equation}

\end{itemize}
\end{lemma}

\begin{lemma}[{\hspace{-0.001em}\cite[Lem.~3.5]{a_BAST_2021a}}]\label{lma:remainder_est}
Let $V\in H^2_{loc}(\VR^d\setminus B)$ satisfy
\begin{equation}
|V(x)|=c_1 |x|^{-m}+\mathcal{O}(|x|^{-m-1}),\quad |\nabla V(x)|=c_2 |x|^{-m-1}+\mathcal{O}(|x|^{-m-2}),
\end{equation}
for $x\in \overline{B}^\mathsf{c}$, $m\in \VN$ and $c_1,c_2>0$ are constants. Then there is a constant $C>0$, such that for $\Gamma\subset \partial \Dsf$ and $\eps>0$ sufficiently small the following estimates hold:

\begin{itemize}

    \item[(i)] $\|V\|_{L^2(\Gamma_\eps^{-1})}\le C \eps^{\frac{2m+1-d}{2}}$.

    \item[(ii)] $|V|_{H^{\frac{1}{2}}(\Gamma_\eps^{-1})}\le C \eps^{\frac{2m+1-d}{2} + \frac12 }$.

    \item[(iii)] $\|\nabla V\|_{L^2(\Gamma_\eps^{-1})^{d}}\le C \eps^{\frac{2m+1-d}{2}+1}$.

\end{itemize}

\end{lemma}

\section{Asymptotics of $-\Delta u_\Omega + \gamma u_\Omega = f_\Omega$ and topological state derivatives}
Throughout this section we let $\Dsf\subset \VR^d$ be a bounded domain with $C^1$ boundary $\partial \Dsf$. We let $\Gamma,\Gamma_N\subset \partial \Dsf$ be connected and  open, such that $\overline\Gamma\cup \overline\Gamma_N = \partial \Dsf$. Henceforth we also assume that if $\gamma=0$, then $\Gamma\ne \emptyset$.

\subsection{First and second topological state derivative of $u_\Omega$}
Given $g\in H^{\frac12}(\Gamma)$ and $h\in L^2(\Gamma_N)$, we consider for an open set $\Omega\subset \Dsf$ the solution $u_\Omega\in H^1(\Dsf)$ to 
\begin{cases2}{eq:state_fOmega}
-\Delta u_\Omega +\gamma u_\Omega & = f_{\Omega} && \quad \text{ in } \Dsf\\
    u_\Omega & = g && \quad \text{ on } \Gamma \\
    \partial_n u_\Omega & = h && \quad \text{ on } \Gamma_N,
\end{cases2}
where $n$ denotes the unit outward pointing normal vector field along $\partial \Gamma_N$ and $\partial_n = n(x) \nabla_x $ denotes the normal derivative. The piecewise constant function $f_\Omega$ is defined by 
\begin{equation}
    f_{\Omega}(x):= \begin{cases}
 f_1 & \text{ for } x\in \Omega \\
 f_2 & \text{ for } x\in \Dsf\setminus \Omega. 
    \end{cases}
\end{equation}

\begin{definition}\label{def:state_derivatives}
Let  $\Omega\subset \Dsf$ be an open set.
    \begin{itemize}
        \item[(a)]
The first topological state derivative of $u_\Omega$ at $x_1\in \Dsf\setminus \partial \Omega$ is defined by 
\begin{equation}
    d_{x_1} u_\Omega := \lim_{\eps\searrow 0} \frac{ u_{\Omega_\eps(x_1)} - u_\Omega}{|B_\eps(x_1)|} . 
\end{equation}
When the set $\Omega$ is fixed we also write $\dot u_{x_1}$ instead of $d_{x_1} u_\Omega$. 
\item[(b)] The second topological state derivative of $u_\Omega$ at $(x_1,x_2)$, $x_1,x_2\in \Dsf\setminus \partial \Omega$ is defined if $x_1\ne x_2$ by 
\begin{equation}
    d_{x_1x_2}^2 u_\Omega := \lim_{\eps\searrow0 } \frac{ d_{x_1} u_{\Omega_\eps(x_1)} - d_{x_1} u_\Omega}{|B_\eps(x_1)|}. 
\end{equation}
When the set $\Omega$ is fixed we also write $\ddot u_{x_1x_2}$ instead of $d^2_{x_1x_2} u_\Omega$.  
\end{itemize}
\end{definition}

\begin{lemma}[\hspace{-0.001em}\cite{a_BAMAST_2023a}]
  Let $\Omega\subset \Dsf$ be an open set. Let $x_1,x_2\in \Dsf\setminus\partial \Omega$ with $x_1\ne x_2$.
    \begin{itemize}
        \item[(a)] The first topological state derivative $\dot u_{x_1} = d_{x_1} u_\Omega$  exists as the strong limit in the space $W^{1,p}_\Gamma(\Dsf)$ for $1\le p<\frac{d}{d-1}$ and is the unique solution to
    \begin{cases2}{eq:ueps_state_derivative_fomega_equation}
        -\Delta \dot u_{x_1} +  \gamma \dot u_\Omega  =& c_{\Omega,x_1} \delta_{x_1} && \quad \text{ in }\Dsf \\
        \dot u_{x_1}  = &0  && \quad\text{ on } \Gamma\\
        \partial_n \dot u_{x_1}  = &0  && \quad\text{ on } \Gamma_N,
\end{cases2}
where $\delta_{x_1}$ denotes the Dirac measure at $x_1$ and $c_{\Omega,x_1}:= \sigma_{\Omega}(x_1) (f_1-f_2)$ is a constant.
\item[(b)]
    The second topological state derivative $\ddot u_{x_1x_2} = d^2_{x_1x_2} u_\Omega$ is zero in  $W^{1,p}_\Gamma(\Dsf)$ for $1\le p < \frac{d}{d-1}$:
\begin{equation}
    d_{x_1x_2}^2 u_\Omega =0. 
\end{equation}
\end{itemize}
\end{lemma}
\begin{proof}
    The first statement (a) is proved in \cite[Thm~2.7]{a_BAMAST_2023a}. The statement (b) follows since for $\eps>0$ small and $x_1,x_2\ne 0$ we have
    \begin{equation}
        \sigma_{\Omega_\eps(x_2)}(x_1)=\sigma_{\Omega}(x_1) 
    \end{equation}
    and therefore $d_{x_1} u_{\Omega_\eps(x_2)} = d_{x_1} u_\Omega$. 
\end{proof}

\subsection{Compound layer expansion of $u_\Omega$ for $\gamma\ge 0$}
We consider for an open set  $\Omega\subset \Dsf$, $x_1\in \Dsf\setminus\partial\Omega$ the function $u_\eps:= u_{\Omega_\eps(x_1)} \in H^1_{\Gamma}(\Dsf)$ which is a solution to:
\begin{cases2}{eq:state_fOmegaeps_gamma}
    -\Delta u_\eps + \gamma u_\eps  & = f_{\Omega_\eps(x_1)} &&\quad  \text{ in } \Dsf\\
    u_\eps & = g && \quad \text{ on } \Gamma \\
    \partial_n u_\eps & = h && \quad \text{ on } \Gamma_N. \\
\end{cases2}
Recall the abbreviation: 
\begin{equation}
    c_{\Omega,x_1}= \sigma_{\Omega}(x_1) (f_1-f_2).
\end{equation}
We recall the definition of the fundamental solution of the Laplace operator $-\Delta$ in dimension two and three is given for $x\ne 0$ by:
\begin{equation}\label{eq:laplace_fundamental}
    E(x)  := \left\{\begin{array}{cc}
    - \frac{1}{2\pi} \ln(|x|) & \text{ for } d=2 \\
     \frac{1}{4\pi} \frac{1}{|x|} & \text{ for } d = 3.
    \end{array}\right. 
\end{equation}
\begin{lemma}
    \begin{itemize}
        \item[(a)]
            Let $d\in \{2,3\}$. For a given $f\in L^2(B)$  the function
\begin{equation}\label{eq:Uf_formula}
    U_f(x):= \int_B f(y)E(x-y)\;dy, \quad x\in \VR^d,
\end{equation}
is in $C^1(\VR^d)$ and $H^2_{loc}(\VR^d)$. Moreover, it is a solution to the equation:
\begin{equation}
    \int_{\VR^d} \nabla U\cdot \nabla \varphi \;dx = \int_B f \varphi\;dx \quad \text{ for all } \varphi \in C^\infty_c(\VR^d). 
\end{equation}
\item[(b)]
For $d=2$, the function $U_f$ has the asymptotic behaviour
\begin{equation}
    U_f(x) = c\ln(|x|) + S^2(x) + S^3(x) + O(|x|^{-3}).
\end{equation}
For $d=3$, the function $U$ has the behaviour
\begin{equation}
    U_f(x) = c\frac{1}{|x|} + S^2(x) + S^3(x) + O(|x|^{-4}). 
\end{equation}
\end{itemize}
\end{lemma}

The functions $\Uk{1},\Uk{2},\vk{1}\vk{2},\Rki{i}{j}$ that we are going to introduce in the following all depend on the point $x_1$ and/or on the open set $\Omega\subset \Dsf$. For notational simplicity we omit the dependence of the shape $\Omega$ and $x_1$.

\begin{definition}
 For $\Omega\subset \Dsf$ open and $x_1\in \Dsf\setminus\partial \Omega$, we define $\Uk{1}:= U_f$ with $f:= c_{\Omega,x_1}$, that is, 
\begin{equation}
    \Uk{1}(x):= c_{\Omega,x_1}\int_B E(x-y)\;dy,
\end{equation}
which solves the equation $-\Delta \Uk{1} = c_{\Omega,x_1} \chi_B$ in $\VR^d$ or equivalently
\begin{equation}
    \int_{\VR^d} \nabla \Uk{1} \cdot \nabla \varphi \;dx = c_{\Omega,x_1} \int_B  \varphi\;dx \quad \text{ for all } \varphi \in C^\infty_c(\VR^d). 
\end{equation}
\end{definition}

\begin{remark}\label{rem:U11_explicit}
    \begin{itemize}
        \item[(a)] The solution $\Uk{1}$ is given by (see page 581 in \cite{a_BAGAST_2021b}) for $d=2$ by 
\begin{equation}\label{eq:Uk1_explicit_2d}
    \Uk{1}(x) =  c_{\Omega,x_1} \begin{cases}
        - \frac{1}{4}(|x|^2-1) \quad & \text{ for } x\in B, \\
        - \frac{1}{2} \ln(|x|) \quad & \text{ for } x\in \VR^2\setminus\overline{B},
    \end{cases}
\end{equation}
and  in dimension $d=3$  by
\begin{equation}\label{eq:Uk1_explicit_3d}
    \Uk{1}(x) =   c_{\Omega,x_1}\begin{cases}
    - \frac{1}{6}(|x|^2-3) \quad & \text{ for } x\in B, \\
        \frac{1}{3} \frac{1}{|x|} \quad & \text{ for } x\in \VR^3\setminus\overline{B}.
    \end{cases}
\end{equation}
    \end{itemize}
\end{remark}

\begin{definition}\label{def:R11_explicit}
  For $\Omega\subset \Dsf$ open and $x_1\in \Dsf\setminus\partial \Omega$,  we define for $x\in \VR^d\setminus\{0\}$: 
\begin{equation}
    \Rki{1}{1}(x):=  c_{\Omega,x_1} \begin{cases} 
        -\frac{1}{2} \ln(|x|) & \quad \text{ for } d=2\\
        \frac{1}{3} \frac{1}{|x|}  & \quad \text{ for } d=3.
    \end{cases}
\end{equation}
\end{definition}

\begin{definition}
  For $\Omega\subset \Dsf$ open and $x_1\in \Dsf\setminus\partial \Omega$, we define $\Uk{2}:= U_f$ with $f:= \gamma (\Rki{1}{1}(y)-\Uk{1}(y))$, that is, 
    \begin{equation}
        \Uk{2}(x)= \int_B \gamma (\Rki{1}{1}(y)-\Uk{1}(y)) E(x-y)\;dy, \quad x \in \VR^d.
    \end{equation}
We have for $|x|\to\infty$:
\begin{equation}
    \Uk{2}(x) = \Rki{2}{1}(x) + \Rki{2}{2}(x) + \begin{cases}
        O(|x|^{-2}) & \text{ for } d=2\\
        O(|x|^{-3}) & \text{ for } d=3
    \end{cases}
\end{equation}
\end{definition}
\begin{remark}
    For $d=2$ we have that $\Rki{2}{1}(x)=c\ln(|x|)$ with some constant $c\in \VR$ and in $d=3$ we have $\Rki{2}{1}(x) = c\frac{1}{|x|}$ with another constant $c\in \VR$. We will discuss the constant in more detail in Subsection~\ref{subsec:explicit_Rki}. 
\end{remark}
\begin{remark}
With an abuse of notation we set for scalars $\lambda\ne0$,  $\Rki{i}{j}(\lambda):= \Rki{i}{j}(\lambda v),v\in \VR^d,|v|=1$.  Then for $d=2$ we have $\Rki{1}{1}(\eps x) = \Rki{1}{1}(x) + \Rki{1}{1}(\eps)$ and for $d=3$ we have $\Rki{1}{1}(\eps x) = \eps^{-1} \Rki{1}{1}(x)$.
\end{remark}

\begin{definition}
Let $\Omega\subset \Dsf$ be open and $x_1\in \Dsf\setminus\partial \Omega$.
    \begin{itemize}
        \item[(a)] We define $\vk{1}\in H^1(\Dsf)$ as the weak solution to 
    \begin{cases2}{eq:vk1_omega}
        -\Delta \vk{1} + \gamma \vk{1} & = -\gamma \Rki{1}{1}(x-x_1) && \quad \text{ in } \Dsf \\
        \vk{1}& = -\Rki{1}{1}(x-x_1) && \quad \text{ on } \Gamma\\
        \partial_n \vk{1} & = -\partial_n \Rki{1}{1}(x-x_1) && \quad \text{ on } \Gamma_N. 
    \end{cases2}
\item[(b)] 
    We define $\vk{2} \in H^1(\Dsf)$ as the weak solution to 
    \begin{cases2}{eq:vk2_omega}
        -\Delta \vk{2} + \gamma \vk{2} & = -\gamma \Rki{2}{1}(x-x_1)  && \quad \text{ in } \Dsf \\
        \vk{2}& = -\Rki{2}{1}(x-x_1) && \quad \text{ on } \Gamma\\
        \partial_n \vk{2} & = -\partial_{n} \Rki{2}{1}(x-x_1)  && \quad \text{ on } \Gamma_N. 
    \end{cases2}
\end{itemize}

\end{definition}

\begin{definition}
    We define the change of variables $T_\eps(x):= \eps x + x_1$. For a given set $A\subset \VR^d$ we set $A_\eps:= T_\eps(A)$ and $A_\eps^{-1}:= T_\eps^{-1}(A)$. 
\end{definition}

\begin{remark}\label{rem:normal_eps_Omega}
    For the bounded domain $\Dsf$ with  $C^1$-boundary $\partial \Dsf$ we denote by $n$ the unit normal along $\partial \Dsf$. According to \cite[Thm. 4.4, page 488]{b_DEZO_2011a} the unit normal vector field $n_\eps$ along $\partial \Dsf_\eps:= T_\eps^{-1}(\Dsf)$ satisfies:
    \begin{equation}
        n_\eps \circ T_\eps^{-1} = \frac{\partial(T_\eps^{-1})n}{|\partial(T_\eps^{-1})n|}.
    \end{equation}
Since $\partial T_\eps^{-1} = \eps^{-1} I$ this leads to $n_\eps = n\circ T_\eps$. 
\end{remark}

\begin{definition}\label{def:Ueps_2d_3d}
    Let $\Omega\subset \Dsf$ be open, $x_1\in \Dsf\setminus\partial \Omega$ and $T_\eps(x) = x_1 + \eps x, \eps\ge 0$.  
    \begin{itemize}
        \item[(a)] For $d=2$ and $\eps >0$, we define
            \begin{align}
                \Uk{1}_\eps &:= \frac{(u_\eps-u_\Omega)\circ T_\eps}{\eps^2}\\\
                \Uk{2}_\eps &:= \frac{\Uk{1}_\eps - (\Uk{1} + \vk{1}\circ T_\eps + \Rki{1}{1}(\eps))}{\eps}\\
                \Uk{3}_\eps &:= \frac{\Uk{2}_\eps - \eps (\Uk{2} + \vk{2}\circ T_\eps + \Rki{2}{1}(\eps))}{\eps}
            \end{align}
        \item[(b)] For $d=3$ and $\eps>0$, we define
 \begin{align}
                \Uk{1}_\eps &:= \frac{(u_\eps-u_\Omega)\circ T_\eps}{\eps^2}\\\
                \Uk{2}_\eps &:= \frac{\Uk{1}_\eps - (\Uk{1} _ \eps + \eps \vk{1}\circ T_\eps) }{\eps}\\
                \Uk{3}_\eps &:= \frac{\Uk{2}_\eps - \eps (\Uk{2} + \eps \vk{2}\circ T_\eps)}{\eps}
            \end{align}
    \end{itemize}

\end{definition}

\begin{lemma}\label{lem:Uf_Deps}
    For $f\in L^2(B)$ let $U_f$ be defined by \eqref{eq:Uf_formula}. Then 
    \begin{equation}
        \int_{\Dsf_\eps^{-1}} \nabla U_f\cdot \nabla \varphi + \gamma U_f \varphi \;dx = \int_B f\varphi \;dx + \int_{(\Gamma_N)_\eps^{-1}} \partial_{n_\eps} U_f\varphi\;ds + \int_{\Dsf_\eps^{-1}} \gamma U_f \varphi \;dx
    \end{equation}
    for all $\varphi \in H^1_{\Gamma_\eps^{-1}}(\Dsf_\eps^{-1})$. Here $n_\eps$ denotes the outward pointing normal vector field along $\partial \Dsf_\eps^{-1}$, which is given by $n_\eps = n\circ T_\eps$ (see Remark~\ref{rem:normal_eps_Omega}), where $n$ denotes the outward pointing unit vector field along $\partial \Dsf$.  
\end{lemma}
\begin{proof}
    This follows by multiplying $-\Delta U_f = f\chi_B$ with $\varphi \in H^1_{\Gamma_\eps^{-1}}(\Dsf_\eps^{-1})$ and partial integration and then adding the term $\int_{\Dsf_\eps^{-1}}\gamma U_f \varphi \;dx$ on both sides.
\end{proof}

\begin{lemma}\label{lem:bound_rk2_Uk2}
Let  $\Omega\subset \Dsf$ be open and  $x_1\in \Dsf\setminus\partial \Omega$. Then are constants $C_1,C_2,C_3>0$, such that for $\eps >0$ small:
\begin{equation}\label{eq:Rki21_U2}
    \|\Rki{2}{1}-\Uk{2}\|_{L^2(\Dsf_\eps^{-1})} \le \begin{cases}
       C_1 + C_2 \sqrt{-\ln(\eps)} & \text{ for } d=2 \\
       C_3                  &\text{ for } d=3.
    \end{cases}
\end{equation}
The constants $C_1,C_2,C_3$ can be chosen to be independent of $\Omega$. 
\end{lemma}
\begin{proof}
    Let $d\in \{2,3\}$. Since $\Rki{2}{1}(x)-\Uk{2}(x) = O(\frac{1}{|x|^{d-1}})$, we find a number $R>0$, such that 
    \begin{equation}
        |\Rki{2}{1}(x)-\Uk{2}(x)|\le C\frac{1}{|x|^{d-1}} \quad \text{ for } |x|>R. 
    \end{equation}
So now we split:
\begin{equation}
    \|\Rki{2}{1}-\Uk{2}\|_{L^2(\Dsf_\eps^{-1})}^2 = \|\Rki{2}{1}-\Uk{2}\|_{L^2(B_R(x_1))}^2 + \|\Rki{2}{1}-\Uk{2}\|_{L^2(\Dsf_\eps^{-1}\setminus B_R(0))}^2. 
\end{equation}
The first term is bounded for fixed $R$ and the second we can estimate further:
\begin{equation}
    \|\Rki{2}{1}-\Uk{2}\|_{L^2(\Dsf_\eps^{-1}\setminus B_R(0))}^2 \le C \|1/|x|^{d-1}\|_{L^2(\Dsf_\eps^{-1}\setminus B_R(0))}^2. 
\end{equation}
Now for some $c>0$ we have $\Dsf_\eps^{-1}\subset B_{c\eps^{-1}}(0)$ for all $\eps>0$, so the last term on the previous inequality we can further estimate 
\begin{equation}
    \|1/|x|^{d-1}\|_{L^2(\Dsf_\eps^{-1}\setminus B_R(0))}^2 \le \|1/|x|^{d-1}\|_{L^2(B_{c\eps^{-1}}(0)\setminus B_R(0))}^2.
\end{equation}
The last integral we can evaluate with polar/spherical coordinates:
\begin{equation}
    \|1/|x|^{d-1}\|_{L^2(B_{c\eps^{-1}}(0)\setminus B_R(0))}^2 =  \begin{cases}
        2\pi \int_{R}^{c\eps^{-1}} \frac{1}{r}\;dr = 2\pi (\ln(c\eps^{-1}) - \ln(R)) & \text{ for } d=2\\
        4\pi  \int_{R}^{c\eps^{-1}} \frac{1}{r^2} \;dr= 4\pi(c^{-1}\eps - R^{-1}) & \text{ for } d=3.
    \end{cases}
\end{equation}
Combining the previous estimates we obtain \eqref{eq:Rki21_U2} with suitable constants $C_1,C_2>0$. They can be chosen independently of $\Omega$ since $\Rki{2}{1},\Uk{2}$ depend on $\Omega$ only via $c_{\Omega,x_1} = (f_1-f_2)\sigma_\Omega(x_1)$ and $|\sigma_\Omega(x_1)|=1$. This concludes the proof.  
\end{proof}

\begin{theorem}\label{thm:remainder_estimate_Deps}
    Let $\Omega\subset \Dsf$ be open and $x_1\in \Dsf\setminus\partial \Omega$. 
    \begin{itemize}
        \item[(a)] There is a constant $C>0$, such that for $\eps >0$:
            \begin{equation}\label{eq:estimate_Uk3eps}
                \|\nabla \Uk{3}_\eps\|_{L^2(\Dsf_\eps^{-1})} + \|\eps \Uk{3}_\eps\|_{L^2(\Dsf_\eps^{-1})} \le \begin{cases}
                     C\eps (-\ln(\eps))^{1/2}) & \text{ for } d=2\\
                     C\eps & \text{ for } d=3
                \end{cases} 
            \end{equation}
            The constant $C$ is independent of $\Omega$. 
        \item[(b)] If $\gamma=0$ and $\Gamma\ne \emptyset$, then $\Uk{2}_\eps=0$, that is,
            \begin{equation}\label{eq:Ueps_explicit}
                \Uk{1}_\eps = 
                \begin{cases}
                    \Uk{1} + \vk{2}\circ T_\eps + \Rki{1}{1}(\eps) & \text{ for } d=2\\
                    \Uk{1} + \eps\vk{2}\circ T_\eps & \text{ for } d=3. 
                \end{cases}
            \end{equation}
    \end{itemize}

\end{theorem}
\begin{proof}
    \emph{ad (a):} We assume for now that $d=2$. The following computation remains valid for $d=3$ where $\Rki{i}{1}(\eps):=0$, $i=1,2$. 
    Writing the equation for $u_\eps-u_0$ gives
\begin{equation}
    \int_\Dsf \nabla (u_\eps-u)\cdot\nabla \varphi + \gamma (u_\eps-u)\varphi\;dx = c_{\Omega,x_1}\int_{B_\eps(x_1)} \varphi\;dx, \quad \text{ for } \varphi\in H_\Gamma^1(\Dsf).
\end{equation}
So changing variables $y=T_\eps(x)$ gives
\begin{equation}
    \int_{\Dsf_\eps^{-1}} \nabla \Uk{1}_\eps\cdot\nabla \varphi + \eps^2 \gamma \Uk{1}_\eps \varphi\;dx = c_{\Omega,x_1} \int_B\varphi\;dx, \quad \text{ for } \varphi \in H^1_{\Gamma_\eps^{-1}}(\Dsf_\eps^{-1}).
\end{equation}
From Lemma~\ref{lem:Uf_Deps} with $f:= c_{\Omega,x_1}$ it follows that $\Uk{1}$ satisfies for  $\varphi \in H^1_{\Gamma_\eps^{-1}}(\Dsf_\eps^{-1})$:
\begin{equation}
    \int_{\Dsf_\eps^{-1}} \nabla \Uk{1}\cdot\nabla \varphi + \eps^2 \gamma \Uk{1}\varphi  \;dx = c_{\Omega,x_i} \int_B \varphi\;dx + \int_{\Dsf_\eps^{-1}} \eps^2 \gamma\Uk{1}\varphi\;dx  + \int_{(\Gamma_N)_\eps^{-1}}  \partial_{n_\eps} \Uk{1} \varphi\;ds. 
\end{equation}
The weak formulation for $\vk{1}$ reads: find $\vk{1}\in H^1(\Dsf)$, such that $\vk{1}=-\Rki{1}{1}(x-x_1)$ on $\Gamma$ and
\begin{equation}
    \int_\Dsf \nabla \vk{1}\cdot \nabla \varphi  + \gamma \vk{1} \varphi\;dx = -\int_\Dsf \gamma \Rki{1}{1}(x-x_1)\varphi \;dx -  \int_{\Gamma_N} \partial_n \Rki{1}{1}(x-x_1)\varphi \;ds.
\end{equation}
Therefore changing variables $y=T_\eps(x)$ and taking into account Lemma~\ref{lem:change_variables_eps} gives for $\varphi \in H^1_{\Gamma_\eps^{-1}}(\Dsf_\eps^{-1})$:
\begin{align}
    \begin{split}
    \int_{\Dsf_\eps^{-1}} \nabla (\vk{1}\circ T_\eps+\Rki{1}{1}(\eps)) \cdot \nabla \varphi  + & \eps^2 \gamma (\vk{1}\circ T_\eps + \Rki{1}{1}(\eps)) \varphi\;dx \\
                                                                                               & = -\eps^2\int_{\Dsf_\eps^{-1}} \gamma \Rki{1}{1}(x)\varphi \;dx -   \int_{(\Gamma_N)_\eps^{-1}} \partial_{n_\eps} \Rki{1}{1}(x)\varphi \;ds.
\end{split}
\end{align}
We note that on $\Gamma$ we have $\vk{1}(T_\eps(x)) +\Rki{1}{1}(\eps)= -\Rki{1}{1}(x-x_1)$. Moreover, $\Rki{1}{1} = \Uk{1}$ on $\VR^2\setminus B$. Therefore for function $\Uk{2}_\eps = \eps^{-1}(\Uk{1}_\eps - \Uk{1} - \vk{1}\circ T_\eps - \Rki{1}{1}(\eps))\in H^1_{\Gamma_\eps^{-1}}(\Dsf_\eps^{-1})$ satisfies:
\begin{align}
    \int_{\Dsf_\eps^{-1}} \nabla  \Uk{2}_\eps\cdot \nabla \varphi + \eps^2 \Uk{2}_\eps\varphi\;dx = \eps \int_{B}\gamma(\Rki{1}{1}(x)-\Uk{1}(x))\varphi\;dx. 
\end{align}
From Lemma~\ref{lem:Uf_Deps} with $f:=  \gamma (\Rki{1}{1}-\Uk{1})$ it follows that $\Uk{2}$ satisfies for $\varphi \in H^1_{\Gamma_\eps^{-1}}(\Dsf_\eps^{-1})$
\begin{align}
    \begin{split}
    \int_{\Dsf_\eps^{-1}} \nabla \Uk{2}\cdot\nabla \varphi + \eps^2 \gamma \Uk{2}\varphi  \;dx =&  \int_{B} \gamma(\Rki{1}{1}-\Uk{1})\varphi\;dx  + \int_{(\Gamma_N)_\eps^{-1}}  \partial_{n_\eps} \Uk{2} \varphi\;dx \\
                                                                                                & + \int_{\Dsf_\eps^{-1}} \eps^2 \gamma \Uk{2}\varphi\;dx.
\end{split}
\end{align}
Proceeding as for $\vk{1}\circ T_\eps$ we see that $\vk{2}\circ T_\eps$ satisfies for all $\varphi\in H^1_{\Gamma_\eps^{-1}}(\Dsf_\eps^{-1})$:
\begin{align}
    \begin{split}
        \int_{\Dsf_\eps^{-1}} \nabla (\vk{2}\circ T_\eps+\Rki{2}{1}(\eps)) \cdot \nabla \varphi  &+ \eps^2\gamma (\vk{2}\circ T_\eps + \Rki{2}{1}(\eps)) \varphi\;dx   \\
                                                                                                 & = -\eps^2\int_{\Dsf_\eps^{-1}} \gamma \Rki{2}{1}(x)\varphi \;dx -   \int_{(\Gamma_N)_\eps^{-1}} \partial_{n_\eps} \Rki{2}{1}(x)\varphi \;ds.
\end{split}
\end{align}
So $\eps\Uk{3}_\eps = \Uk{2}_\eps - \eps(\Uk{2}-\vk{2}\circ T_\eps-\Rki{2}{1}(\eps))$ satisfies:  $\eps \Uk{3}_\eps = -\eps(\Uk{2} - \Rki{2}{1})$ on $\Gamma_\eps^{-1}$ and 
\begin{align}
    \begin{split}
    \int_{\Dsf_\eps^{-1}} \nabla  (\eps \Uk{3}_\eps)\cdot \nabla \varphi + \eps^2 (\eps \Uk{3}_\eps)\varphi\;dx =& \eps^3 \int_{\Dsf_\eps^{-1}} (\Rki{2}{1}-\Uk{2})\varphi\;dx \\
                                                                                                                 & +  \eps\int_{(\Gamma_N)_\eps^{-1}} \partial_{n_\eps} (\Rki{2}{1}(x)-\Uk{2}(x))\varphi \;ds
\end{split}
\end{align}
for all $\varphi \in H^1_{\Gamma_\eps^{-1}}(\Dsf_\eps^{-1})$.
Now testing the previous equation with $\varphi = \eps \Uk{3}_\eps - Z_{\Gamma_\eps^{-1}}(\eps \Uk{3}_\eps) \in H^1_{\Gamma_\eps^{-1}}(\Dsf_\eps^{-1})$ and applying Lemma~\ref{lem:scaling_inequality}, (d),(e) and (f) yields
\begin{align}\label{eq:Ueps_3}
    \begin{split}
    \|\nabla (\eps \Uk{3}_\eps)\|_{L^2(\Dsf_\eps^{-1})} + \eps \|(\eps \Uk{3}_\eps)\|_{L^2(\Dsf_\eps^{-1})} \le&  C\big( \eps^2\|\Rki{2}{1}-\Uk{2}\|_{L^2(\Dsf_\eps^{-1})}  + \eps^{\frac12}\|\partial_{n_\eps}(\Rki{2}{1}-\Uk{2})\|_{L^2((\Gamma_N)_\eps^{-1})} \\
                                                                                                               & + \eps(\eps^{\frac12} \|\Rki{2}{1}-\Uk{2}\|_{L^2(\Gamma_\eps^{-1})} + |\Rki{2}{1}-\Uk{2}|_{H^{\frac12}(\Gamma_\eps^{-1})})\big). 
    \end{split}
\end{align}
It is readily checked that \eqref{eq:Ueps_3} also holds for $d=3$ with the respective definition of $\Uk{i},\Uk{i}_\eps$, $i=1,2$ (cf. Definition~\ref{def:Ueps_2d_3d}, (b)).  Now from from Lemma~\ref{lem:bound_rk2_Uk2} and Lemma~\ref{lma:remainder_est} applied to $V(x):= \Rki{2}{1}(x)-\Uk{2}(x)$, with $m=1$ for $d=2$ and $m=2$ for $d=3$, we conclude:
\begin{align}
    \|\Rki{2}{1}-\Uk{2}\|_{L^2(\Dsf_\eps^{-1})} &\le 
    \begin{cases}
        C_1\sqrt{-\ln(\eps)} + C_2 & \text{ for } d=2\\
        C_3 & \text{ for } d=3
    \end{cases}\\
        \|\Rki{2}{1}-\Uk{2}\|_{L^2((\Gamma_N)_\eps^{-1})}& \le 
        \begin{cases}
            C\eps^{\frac12} & \text{ for } d=2\\
            C\eps & \text{ for } d=3
        \end{cases} \\
        |\Rki{2}{1}-\Uk{2}|_{H^{\frac12}((\Gamma_N)_\eps^{-1})} & \le 
            \begin{cases}
        C \eps & \text{ for }d=2\\
        C\eps^{\frac32} & \text{ for } d=3
            \end{cases} \\ 
        \|\partial_{n_\eps}(\Rki{2}{1}-\Uk{2})\|_{L^2((\Gamma_N)_\eps^{-1})} & \le 
        \begin{cases}
    C \eps^{\frac32} & \text{ for } d=2 \\
    C \eps^2 & \text{ for } d=3.
        \end{cases}
\end{align}
So for $d=2$ using the previous inequalities to estimate the right hand side of \eqref{eq:Ueps_3} and dividing by $\eps$ we obtain the claimed estimate noting that $\eps \sqrt{-\ln(\eps)}$ converges slower to zero than $\eps$. Similarly for  $d=3$ we obtain the bound $C\eps$ as $\Rki{2}{1}$ does not produce a logarithmic term. The constant $C$ can be chosen independently of $\Omega$, since $\Rki{2}{1}$ and $\Uk{2}$ both depend on $\Omega$ only through $c_{\Omega,x_1} = (f_1-f_2)\sigma_{\Omega}(x_1)$ and $|\sigma_{\Omega}(x_1)|=1$.  This finishes the proof of (a). 

\emph{ad (b):} This has been shown in \cite{a_BAGAST_2021b} on page 581. It also follows from the proof of (a) by noting that $\gamma=0$ implies $\Uk{2}=0$, thus $\Rki{2}{1}=0,\vk{2}=0$ and thus \eqref{eq:Ueps_3} implies $\eps \Uk{3}_\eps=0$ on $\Dsf_\eps^{-1}$ which is equivalent to \eqref{eq:Ueps_explicit}.
\end{proof}

\begin{corollary}\label{cor:expansion_remainder_D}
    Let $\Omega\subset \Dsf$ be open and $x_1\in \Dsf\setminus\partial \Dsf$. 
    \begin{itemize}
        \item[(a)] 
            For dimension $d=2$ we have that $\Ce_\Omega(\eps):= \eps^4 \Uk{3}_\eps\circ T_\eps^{-1}$ satisfies:
\begin{align}\label{eq:ueps_expansion_D_2d}
    \begin{split}
    u_\eps = u_\Omega &+ \eps^2(\Uk{1}\circ T_\eps^{-1} + \vk{1}\circ T_\eps + \Rki{1}{1}(\eps)) \\
                      & + \eps^4(\Uk{2}\circ T_\eps^{-1} + \vk{2}\circ T_\eps + \Rki{2}{1}(\eps)) + \Ce_\Omega(\eps)  \quad \text{ a.e. in } \Dsf
\end{split}
\end{align}
with $\|\Ce_\Omega(\eps)\|_{H^1(\Dsf)}=O(\eps^5\sqrt{-\ln(\eps)})$ uniformly in $\Omega$. 
\item[(b)] For dimension $d=3$, we have that $\Ce_\Omega(\eps):= \eps^4 \Uk{3}_\eps\circ T_\eps^{-1}$ satisfies
\begin{align}\label{eq:ueps_expansion_D_3d}
    \begin{split}
    u_\eps = u_\Omega &+ \eps^2(\Uk{1}\circ T_\eps^{-1} + \eps \vk{1}\circ T_\eps )\\
                      &+ \eps^4(\Uk{2}\circ T_\eps^{-1} + \eps \vk{2}\circ T_\eps ) + \Ce_\Omega(\eps) \quad \text{ a.e. in } \Dsf
\end{split}
\end{align}
with $\|\Ce_\Omega(\eps)\|_{H^1(\Dsf)}=O(\eps^{\frac{11}{2}})$ uniformly in $\Omega$. 
\item[(c)] If $\gamma=0$ and $\Gamma\ne \emptyset$, then $\Uk{2}_\eps=0$ and thus
    \begin{equation}
        u_\eps = u_\Omega + \begin{cases}
            \eps^2(\Uk{1}\circ T_\eps^{-1} + \vk{1} + \Rki{1}{1}(\eps)) & \text{ for } d=2\\
            \eps^2(\Uk{1}\circ T_\eps^{-1} + \eps \vk{1}) & \text{ for } d=3. 
        \end{cases}
    \end{equation}
\end{itemize}
\end{corollary}
\begin{proof}
    \emph{ad (a) and (b):}  Set  $\Ce_\Omega(\eps):= \eps^4\Uk{3}_\eps\circ T_\eps^{-1}$, where $\Uk{3}_\eps$ is defined as in Definition~\ref{def:Ueps_2d_3d}, (a).  Then \eqref{eq:ueps_expansion_D_2d} is satisfied. Now we compute changing variables:
\begin{align}
    \int_\Dsf (\eps^4\Uk{3}_\eps\circ T_\eps^{-1})^2\;dx &= \eps^{6+d}\int_{\Dsf_\eps^{-1}} (\eps \Uk{3}_\eps)^2\;dx \\
    \int_\Dsf |\eps^4\nabla(\Uk{3}_\eps\circ T_\eps^{-1})|^2\;dx &= \eps^{6+d}\int_{\Dsf_\eps^{-1}} |\nabla\Uk{3}_\eps|^2\;dx.
\end{align}
Now taking square roots on both sides and using Theorem~\ref{thm:remainder_estimate_Deps}, we obtain
\begin{align}
    \|\Ce_\Omega(\eps)\|_{H^1(\Dsf)}  & \le \eps^4(\|\eps \Uk{3}_\eps\|_{L^2(\Dsf_\eps^{-1})} + \|\nabla \Uk{3}_\eps\|_{L^2(\Dsf_\eps^{-1})})
                                        \stackrel{\eqref{eq:estimate_Uk3eps}}{\le } \begin{cases}
        C\eps^5\sqrt{-\ln(\eps)} & \text{ for } d=2 \\
        C\eps^{\frac{11}{2}} & \text{ for } d=3,
    \end{cases}
\end{align}
where in the last step we used \eqref{eq:estimate_Uk3eps}.
The constant $C$ in the previous estimate depends on via $c_{\Omega,x_1} = \sigma_\Omega(x_1)(f_1-f_1)$ on the shape. Therefore this constant is bounded for small perturbations of the shape $\Omega$.

\emph{ad (c)} This is a direct consequence of Theorem~\ref{thm:remainder_estimate_Deps} part (b).

\end{proof}

\begin{remark}
We note that for the three dimensional case $d=3$ in (b), the remainder estimate is of order smaller than $6$ (corresponding to $|B_\eps|^2$). However, for the expansion of the cost functional later on its remainder estimates this is sufficient. 
\end{remark}

\subsection{The special case of $\gamma=0$: Asymptotics of $-\Delta u_\Omega=f_\Omega$} 

We consider the special case $\gamma=0$ with $|\Gamma|\ne 0$. That is for $x_1\in \Dsf\setminus\partial \Omega$ let $u_\eps=u_{\Omega_\eps(x_1)}\in H^1_0(\Dsf)$, $\eps>0$, be the solution to 
\begin{cases2}{eq:state_fOmegaeps}
    -\Delta u_\eps  & = f_{\Omega_\eps(x_1)} && \quad  \text{ in } \Dsf\\
    u_\eps & = g && \quad \text{ on } \Gamma \\
    \partial_n u_\eps & = h && \quad \text{ on } \Gamma_N.
\end{cases2}
 In this case the expansion simplifies considerably as discussed next. We have shown that in this case we have
\begin{equation}\label{eq:explicit_23d}
    u_\eps = u_\Omega +\begin{cases}
 \eps^2 \Uk{1}\circ T_\eps^{-1} + \vk{1} + \Rki{1}{1}(\eps) & \text{ for } d=2\\
 \eps^2 (\Uk{1}\circ T_\eps^{-1} + \eps\vk{1}) & \text{ for } d=3.
    \end{cases}
\end{equation}

\subsubsection{Asymptotics in dimension $d=2$}
The explicit formula for $\Uk{1}$ given in \eqref{eq:Uk1_explicit_2d} and the definition $\Rki{1}{1}(x)=-\frac12c_{\Omega,x_1}\ln(|x-x_1)$ (cf. Definition~\ref{def:R11_explicit}) show that \eqref{eq:explicit_23d} can be written for $d=2$ and $x_1\in \Dsf\setminus\partial \Omega$ as follows:
\begin{align}\label{eq:expansion_explicit_fomega}
    u_\eps & = u_0 + \eps^2 m_{\Omega,x_1} \begin{cases}
          -\frac12( \eps^{-2} |x-x_1|^2 - 1)  + \hat v  +  \ln(\eps) & \quad \text{ for } x\in B_\eps(x_1)\\
         -\ln(|x-x_1|) + \hat v   & \quad \text{ for } x\in \Dsf\setminus\overline{B_\eps(x_1)},  
    \end{cases}
\end{align}
where $m_{\Omega,x_1} = \sigma_{\Omega}(x_1) (f_1-f_2)/2$ and $\hat v\in H^1(\Dsf)$ is the solution to 
\begin{cases2}{eq:mixed_v_fomega}
    -\Delta \hat v & = 0   &&\quad \text{ in } \Dsf\\
    \hat v & = \ln(|x-x_1|) && \quad \text{ on } \Gamma_D \\
    \partial_n \hat v & = \partial_{n}\ln(|x-x_1|) && \quad \text{ on } \Gamma_N.
\end{cases2}
Note that $\vk{1}=m_{\Omega,x_1}\hat v$. As pointed out in \cite{a_BAGAST_2021b} there is an intricate relation between the first order topological expansion, the compound layer expansion and the first topological state derivative.
\begin{lemma}
We have for $x_1\in \Dsf\setminus\partial \Omega$ the topological state derivative $d_{x_1} u_\Omega = |B|^{-1}(\Rki{1}{1}(x-x_1) + \vk{1})$ can be written as:
\begin{equation}
    d_{x_1}u_\Omega =  |B|^{-1}m_{\Omega,x_1}(-\ln|x-x_1| + \hat v),
\end{equation}
where $|B|=\pi$ and 
\begin{equation}\label{eq:ueps_state_derivative_fomega}
    u_\eps - u_0 = |B_\eps| d_{x_1}u_\Omega \quad \text{ on } \Dsf\setminus \bar B_\eps(x_1).
\end{equation}
\end{lemma}
\begin{proof}
By comparing \eqref{eq:expansion_explicit_fomega} on $\Dsf\setminus\overline B_\eps(x_1)$ with $u_\eps = u_0 +|B_\eps| d_{x_1}u_\Omega + o(1)$  the first formula follows. From this the second one immediately follows as well.
\end{proof}

\subsubsection{Asymptotics in dimension $d=3$}
For $d=3$ we have $\Rki{1}{1}(x) = c_{\Omega,x_1} |x-x_1|^{-1}$ (cf., Definition~\ref{def:R11_explicit}). Hence recalling the explicit formula \eqref{eq:Uk1_explicit_3d} for $\Uk{1}$, we see that  equation \eqref{eq:explicit_23d} can be written for $x_1\in \Dsf\setminus \partial \Omega$ as:
\begin{align}
    u_\eps  = u_0 + \eps^3 m_{\Omega,x_1} \begin{cases}
         -\frac12 (\eps^{-2}|x-x_1|^2 -3) + \hat v    \quad & \text{ for } x\in B_\eps(x_1)\\
           \frac{1}{|x-x_1|} + \hat v \quad & \text{ for } x\in \Dsf\setminus\overline{B_\eps(x_1)},
\end{cases}
\end{align}
where $m_{\Omega,x_1} = \sigma(x_1) (f_1-f_2)/3$ and $\hat v\in H^1(\Dsf)$ is the unique solution to 
\begin{cases2}{eq:mixed_v_fomega_3d}
    -\Delta \hat v & = 0   &&\quad \text{ in } \Dsf\\
    \hat v & = -\frac{1}{|x-x_1|} && \quad \text{ on } \Gamma_D \\
    \partial_n \hat v & = -\partial_{n}\frac{1}{|x-x_1|} && \quad \text{ on } \Gamma_N.
\end{cases2}

\begin{lemma}
We have for $x_1\in \Dsf\setminus\partial \Omega$ that the topological state derivative $d_{x_1} u_\Omega = |B|^{-1}(\Rki{1}{1}(x-x_1) + \vk{2})$ can be written as:
\begin{equation}
    d_{x_1}u =  |B|^{-1}m_{\Omega,x_1}\left(\frac{1}{|x-x_1|} + \hat v\right),
\end{equation}
where $|B|=\frac{4\pi}{3}$ and in particular:
\begin{equation}\label{eq:ueps_state_derivative_fomega_3d}
    u_\eps - u_0 = |B| d_{x_1}u_\Omega \quad \text{ on } \Dsf\setminus\overline B_\eps(x_1).
\end{equation}
\end{lemma}
\begin{proof}
This follows again by comparing \eqref{eq:state_fOmegaeps} with $u_\eps = u_0 +|B_\eps| d_{x_1}u_\Omega + o(1)$ on $\Dsf\setminus\overline B_\eps(x_1)$. 
\end{proof}

\section{Second topological state derivative at the same point $x_1$ }
The second topological state derivative  $d_{x_1x_2}^2 u_\Omega$ has been introduced in Definition~\ref{def:state_derivatives} under the assumption $x_1\ne x_2$. In fact if we were to consider the same definition for $x_1=x_2$ and compare $d_{x_1} u_{\Omega\setminus B_\eps(x_1)}$ with $d_{x_1} u_\Omega$ for, e.g.,  $x_1\in \Omega$ and note $1 = \sigma_{\Omega\setminus B_\eps(x_1)}$ and $  \sigma_\Omega(x_1)=-1$ for all $\eps>0$, then 
\begin{equation}
    \int_\Dsf |B_\eps(x_1)|^{-1} (d_{x_1} u_{\Omega_\eps(x_1)}-d_{x_1}u_\Omega)\varphi \;dx = -|B_\eps(x_1)|^{-1}2(f_1-f_2) \varphi(x_1) 
\end{equation}
for $\varphi\in W^{1,p}_\Gamma(\Dsf), \; p\in [1,\frac{d}{d-1})$ and the right hand side goes to infinity when $\eps$ goes to zero. 

It is therefore to be expected that the second topological state derivative of $u_\Omega$ at the point $x_1$ cannot be a function, but is at best a distribution. Before we can introduce the second order topological state derivative at a point $x_1$, we need to discuss the term $\Rki{2}{1}$ which plays a crucial role in its definition. 

\subsection{Explicit constant of $\Rki{2}{1}$}\label{subsec:explicit_Rki}
For convenience we recall $c_{\Omega,x_1}=\sigma_\Omega(x_1)(f_1-f_2)$. The next lemma provides the constants $c_1,c_2$ in front of 
$\Rki{2}{1}(x) = c_1\ln(|x|)$ for $d=2$ and $\Rki{2}{1}(x) = c_2/|x|$ for $d=3$. These constants play a crucial role in relating the compound layer expansion to the second order topological state derivative defined in the next section. 
\begin{lemma}
    \begin{itemize}
        \item[(a)] 
  In dimension $d=2$, we have
    \begin{equation}
        \Rki{2}{1}(x) = \gamma K_{\Omega,x_1}\ln(|x|), \quad x\in \VR^2\setminus\{0\}
    \end{equation}
    with the constant:
    \begin{equation}
        K_{\Omega,x_1} := c_{\Omega,x_1}
            \frac{1}{2\pi} \int_B \frac12 \ln(|x|) -\frac14(|x|^2-1)\;dx.
    \end{equation}
    Since $\int_B \ln(|x|)\;dx = \int_B |x|^2-1\;dx=-\frac{\pi}{2}$:
    \begin{equation}\label{eq:K_Omega_x1_2d}
        K_{\Omega,x_1} = c_{\Omega,x_1}\frac{1}{8 \pi}\int_B|x|^2-1\;dx = -c_{\Omega,x_1}\frac{1}{16}. 
    \end{equation}
\item[(b)] In dimension $d=3$, we have
    \begin{equation}
        \Rki{2}{1}(x) = \gamma K_{\Omega,x_1} \frac{1}{|x|}, \quad x\in \VR^3\setminus \{0\}
    \end{equation}
    with the constant:
    \begin{equation}
        K_{\Omega,x_1} := c_{\Omega,x_1}\frac{1}{4\pi} \int_B \frac{1}{3|x|}+\frac16(|x|^2-3)\;dx.
    \end{equation}
    Here the constant computes as follows:
    \begin{equation}\label{eq:K_Omega_x1_3d}
        K_{\Omega,x_1} = c_{\Omega,x_1}\frac{1}{30}.
    \end{equation}
\end{itemize}
\end{lemma}
\begin{proof}
    The fundamental theorem of calculus applied $g:t\mapsto \int_B(\Rki{1}{1}(y)-\Uk{1}(y))E(x-ty)\;dy$  on $[0,1$] shows that for $x\ne0$:
 \begin{align}
     \Uk{2}(x) =&  \gamma E(x) \int_B (\Rki{1}{1}(y)-\Uk{1}(y))\;dy \\
                & - \gamma \int_0^1\int_B (\Rki{1}{1}(y)-\Uk{1}(y))\partial_y E(x-ty)\;dy\;dt.
\end{align}
This shows that for $x\ne 0$:
\begin{equation}
    \Rki{2}{1}(x) = \gamma E(x) \int_B \Rki{1}{1}(y)-\Uk{1}(y)\;dy.
\end{equation}
The integral can be explicitly evaluated using the explicit formula for $\Uk{1}$ given in Remark~\ref{rem:U11_explicit} and using polar/spherical coordinates in dimension two/three, respectively.
\end{proof}

\begin{remark}\label{rem:const_int_x2_1_KOmega}
    One readily computes using polar/spherical coordinates:
\begin{equation}\label{eq:int_x2_1_two}
     \frac14\int_B |x|^2-1\;dx = \begin{cases}
        -\frac{\pi}{8} & \text{ for } d=2 \\
        -\frac{2\pi}{15} & \text{ for } d=3.
    \end{cases}
\end{equation}
Therefore we get from \eqref{eq:K_Omega_x1_2d},\eqref{eq:K_Omega_x1_3d} the important relation:
\begin{equation}
    c_{\Omega,x_1} \frac14 \int_B |x|^2-1\;dx = \begin{cases}
        2\pi K_{\Omega,x_1} & \text{ for } d=2\\
        -4\pi K_{\Omega,x_1} & \text{ for } d=3.
    \end{cases}
\end{equation}

\end{remark}

\subsection{The second topological state derivative $d_{x_1}^2u_\Omega$ at $x_1$}
The first order topological state derivative $d_{x_1}u_\Omega$ of $u_\Omega$ has been introduced in Definition~\ref{def:state_derivatives}. Now we define the second order topological state derivative at the point $x_1$. We start with a motivation for the definition.

\subsubsection{Motivation for the second derivative}
We consider $d=2$ for the following discussion. Recall that the first topological state derivative of $u_\Omega$ at a point $x_1\in \Dsf\setminus \partial \Omega$ is given by
\begin{equation}
    d_{x_1} u_\Omega(x) = |B|^{-1}(\Rki{1}{1}(x-x_1) + \vk{1}(x)).
\end{equation}
The compound layer expansion can be written on $\Dsf\setminus \bar B_\eps(x_1)$ as:
\begin{equation}\label{eq:expansion_layer_motivation}
    u_\eps  = u_0 + |B_\eps|d_{x_1}u_\Omega   + \frac{|B_\eps|^2}{2} \left(\frac{2}{|B|^2}(\Uk{2}\circ T_\eps^{-1}+\vk{2} + \Rki{2}{1}(\eps))\right) + o(\eps^4). 
\end{equation}
This looks like a Taylor expansion in the volume $|B_\eps|$. Now recall that for large $|x|$ we have $\Uk{2}(x) \approx \Rki{2}{1}(x)$ and therefore for $x\in \Dsf\setminus \{x_1\}$ and small $\eps>0$
\begin{equation}
    \Uk{2}\circ T_\eps^{-1}(x)\approx \Rki{2}{1}(T^{-1}_\eps(x)) = \Rki{2}{1}(x-x_1)- \Rki{2}{1}(\eps)
\end{equation}
which means that the "quadratic term" corresponding to $\frac{|B_\eps|^2}{2}$ in \eqref{eq:expansion_layer_motivation} approximates as follows
\begin{equation}
    \frac{2}{|B|^2}\left(\Uk{2}\circ T_\eps^{-1}(x)+\vk{2}(x) + \Rki{2}{1}(\eps)\right) \approx \frac{2}{|B|^2}\left(\Rki{2}{1}(x-x_1) + \vk{2}(x)\right).
\end{equation}
Therefore it is expected, if \eqref{eq:expansion_layer_motivation} is viewed as a Taylor-like expansion in the volume $|B_\eps|$, that the second topological state derivative of $u_\Omega$ at $x_1$ should be of the form:
$2|B|^{-2}(\Rki{2}{1}(x-x_1) + \vk{2}(x))$. The discussion for $d=3$ is analogous and we proceed in the next section with the formal definition of the second topological state derivative at $x_1$. 

\subsubsection{Definition for dimension $d=2$}

The discussion of the previous section motivates the following definition.

\begin{definition}\label{def:second_order_state_derivative}
    The second order topological state derivative of $u_\Omega$ at $x_{1}\in \Dsf\setminus\partial \Omega$ is defined by:
\begin{equation}
    d_{x_1}^2 u_{\Omega}(x):= \mathfrak d^2_{x_1} u_{\Omega}(x)  - \frac{2}{|B|^2} 2\pi K_{\Omega,x_1} \delta_{x-x_1}, \quad x\in \Dsf\setminus\{x_1\}, 
\end{equation}
where $\delta_{x-x_1}$ is the delta-distribution at $x_1$ and
\begin{equation}
    \mathfrak d^2_{x_1} u_{\Omega}(x) :=     \frac{2}{|B|^2} (\Rki{2}{1}(x-x_1) + \vk{2}(x)), \quad \text{ for } x\in \Dsf\setminus \{x_1\}. 
\end{equation}
\end{definition}
Here we understand the delta-distribution for $f\in C(\Dsf)$ as follows:
\begin{equation}
    \int_\Dsf f \delta_{x-x_1} dx = \int_\Dsf f d\delta_{x_1} = f(x_1). 
\end{equation}

\begin{lemma}
    The function $z=\mathfrak d^2_{x_1} u_{\Omega}\in W^{1,p}_\Gamma(\Dsf)$, $1\le p < \frac{d}{d-1}$ satisfies:
\begin{cases2}{eq:frak_d}
    -\Delta z   & + \gamma z = -\frac{2}{|B|^2}2\pi \gamma K_{\Omega,x_1} \delta_{x_1} && \quad  \text{ in } \Dsf\\
    z & = 0 && \quad \text{ on } \Gamma \\
    \partial_n z & = 0 && \quad \text{ on } \Gamma_N.
\end{cases2}
\end{lemma}
\begin{proof}
    Notice that $\Rki{2}{1}(x) = \gamma K_{\Omega,x_1}\ln(|x|) =   -2\pi\gamma K_{\Omega,x_1} E(x)$ with $E(x)=-\frac{1}{2\pi}\ln(|x|)$. Therefore 
    \[
        -\Delta \Rki{2}{1}(x-x_1)  = -2\gamma \pi K_{\Omega,x_1}(-\Delta E(x-x_1))= -2\gamma \pi K_{\Omega,x_1}\delta_{x_1}.
    \]
    From this and the definition of $\vk{2}$ the result follows. 
\end{proof}

\begin{remark}
Notice that we can write the second state derivative as
\begin{equation}
    d_{x_1}^2u_\Omega(x)= \mathfrak d_{x_1}^2u_\Omega(x) + (-\Delta + \gamma)(\mathfrak d_{x_1}^2u_\Omega)(x).
\end{equation}
\end{remark}

\subsubsection{Definition for dimension $d=3$}
Following the motivation of the case $d=2$ we define the second topological state derivative as follows:
\begin{definition}
    The second order topological state derivative of $u_\Omega$ at $x_{1}\in \Dsf\setminus\partial \Omega$ is defined by:
\begin{equation}
    d_{x_1}^2 u_{\Omega}(x):= \mathfrak d^2_{x_1} u_{\Omega}(x)  + \frac{2}{|B|^2} 4\pi K_{\Omega,x_1} \delta_{x-x_1}, \quad x\in \Dsf\setminus\{x_1\}, 
\end{equation}
where $\delta_{x-x_1}$ is the delta-distribution at $x_1$ and
\begin{equation}
    \mathfrak d^2_{x_1} u_{\Omega}(x) :=     \frac{2}{|B|^2} (\Rki{2}{1}(x-x_1) + \vk{2}(x)) \quad \text{ for } x\in \Dsf\setminus \{x_1\}. 
\end{equation}
\end{definition}
\begin{lemma}
    The function $z=\mathfrak d^2_{x_1} u_{\Omega}\in W^{1,p}_\Gamma(\Dsf)$, $1\le p < \frac{d}{d-1}$ satisfies:
\begin{cases2}{eq:frak_d_3d}
    -\Delta z   & + \gamma z = \frac{2}{|B|^2}4\pi \gamma K_{\Omega,x_1} \delta_{x_1} && \quad  \text{ in } \Dsf\\
    z & = 0 && \quad \text{ on } \Gamma \\
    \partial_n z & = 0 && \quad \text{ on } \Gamma_N.
\end{cases2}
\end{lemma}
\begin{proof}
    Notice that $\Rki{2}{1}(x) = \gamma K_{\Omega,x_1}\frac{1}{|x|} =  4\pi \gamma K_{\Omega,x_1} E(x)$ with $E(x)=\frac{1}{4\pi|x|}$. Therefore for $x\ne x_1$:
    \[
        -\Delta \Rki{2}{1}(x-x_1) = (4 \pi \gamma K_{\Omega,x_1})(-\Delta E(x-x_1)) =  4 \pi \gamma K_{\Omega,x_1}\delta_{x_1}.
    \]
    From this and the definition of $\vk{2}$ the result follows. 
\end{proof}

\begin{remark}
Notice that we can write the second state derivative as
\begin{equation}
    d_{x_1}^2u_\Omega(x)= \mathfrak d_{x_1}^2u_\Omega(x) + (-\Delta + \gamma)(\mathfrak d_{x_1}^2u_\Omega)(x).
\end{equation}
\end{remark}

\subsection{Formal topological calculus with the state derivatives in dimension $d=2$}
We briefly discuss how the first and second topological state derivative can be used to formally obtain the correct second 
topological derivative of a cost functional $J(\Omega)$ in dimension $d=2$. In the following we assume  $u_r\in H^1(\Dsf)\cap C(\Dsf)$. Let $\Omega\subset \Dsf$ be an open set. For $x_1,x_2\in \Dsf\setminus \partial \Omega$, we set
\begin{equation}
d_{x_1x_2}^2u_\Omega:= \begin{cases}
    d_{x_1x_2}^2u_\Omega & \text{ for } x_1\ne x_2  \quad \text{(Definition~\ref{def:first_and_second})}\\
    d_{x_1}^2u_\Omega & \text{ for } x_1=x_2 \quad \text{(Definition~\ref{def:second_order_state_derivative})}.
\end{cases}
\end{equation}

We now illustrate in a few examples that those second order derivatives behave like usual  derivatives in Banach spaces, obeying the chaing and product rule in dimension $d=2$.

\subsubsection{Second topological derivative of $J_1(\Omega) = \int_\Dsf (u-u_r)^2 \;dx$}
The first order topological derivative can be computed  at $x_1\in \Dsf\setminus\partial \Omega$:
\begin{equation}
DJ(\Omega)(x_1) =    d_{x_1}J_1(\Omega) = \int_\Dsf 2(u-u_r) d_{x_1} u_\Omega\;dx.
\end{equation}
Applying now $d_{x_2}$ for $x_2\in \Dsf\setminus\partial \Omega$ (note we set $d_{x_2}d_{x_1}:= d_{x_1x_2}$):
\begin{align}
    d_{x_1 x_2}^2J_1(\Omega) =  \int_\Dsf 2 (d_{x_1}u_\Omega) (d_{x_2}u_\Omega) + 2(u_\Omega-u_r)d^2_{x_1x_2}u_\Omega \;dx.
\end{align}
We will prove later in Theorem~\ref{thm:expansion_J1_fOmega} that this formula is indeed correct in dimension $d=2$, that means, 
\[
    D^2J(\Omega)(x_1,x_2) = d_{x_2x_1}^2J_1(\Omega) \quad \text{ for all } x_1,x_2\in \Dsf\setminus\partial \Omega.
\]
 For dimension $d=3$ the formula is not correct any longer.

\subsubsection{Second topological derivative of $J_2(\Omega)= \int_{\Gamma_N} (u-u_r)^2\;ds$}
The first topological derivative computes again like $x_1\in \Dsf\setminus\partial \Omega$
\begin{equation}
    DJ(\Omega)(x_1) = d_{x_1} J_2(\Omega) = \int_{\Gamma_N} 2(u-u_r) d_{x_1}u_\Omega\;ds. 
\end{equation}
Applying formally $d_{x_1}$ again we obtain:
\begin{align}
    \begin{split}
    d_{x_1}^2J_2(\Omega) &=  \int_{\Gamma_N} 2 (d_{x_1}u)^2 \;ds + \int_{\Gamma_N} 2(u-u_r) d_{x_1}^2\;ds \\
                         & = \int_{\Gamma_N} 2 (d_{x_1}u)^2 \;ds + \int_{\Gamma_N} 2(u-u_r) \mathfrak d_{x_1}^2u_\Omega\;ds  - \frac{2}{|B|^2}2\pi K_{\Omega,x_1}\underbrace{\int_{{\Gamma_N}} 2(u-u_r)\delta_{x-x_1}\;ds}_{=0}.
\end{split}
\end{align}
Indeed we will show later that (see Theorem~\ref{thm:expansion_J2_fOmega}, (a)):
\begin{equation}
   D^2J_2(\Omega)(x_1,x_1) =  \int_{\Gamma_N} 2 (d_{x_1}u)^2 \;ds + \int_{\Gamma_N} 2(u-u_r) \mathfrak d_{x_1}^2u_\Omega\;ds.
\end{equation}

\subsubsection{Second topological derivative of $J_3(\Omega)= \int_\Dsf u\;dx$}
The first derivative reads
\begin{equation}
    DJ_3(\Omega)(x_1) = d_{x_1} J(\Omega) = \int_\Dsf d_{x_1}u_\Omega\;dx
\end{equation}
and the second reads
\begin{equation}
    d_{x_1}^2 J(\Omega) = \int_{\Dsf} d_{x_1}^2 u_\Omega \;dx,
\end{equation}
which is also correct and indeed $D^2J(\Omega)(x_1,x_1) = d_{x_1}^2 J(\Omega)$. We do not prove this, but it can be checked with the same proof strategy as in Theorem~\ref{thm:expansion_J1_fOmega}.

\subsection{The second topological state derivative at $x_1$ in dimension $d=3$}
In dimension $d=3$ we cannot obtain the second derivative $D^2J(\Omega)(x_1,x_1)$ by applying $d_{x_1}$ twice to $J(\Omega)$. However, we sill have for instance for
\begin{equation}
    J_1(\Omega) = \int_\Dsf (u-u_r)^2\;dx
\end{equation}
that (see Theorem~\ref{thm:expansion_J1_fOmega}, (b) below)
\begin{equation}
    D^2J(\Omega)(x_1,x_1) = \int_\Dsf 2(d_{x_1}u_\Omega)^2\;dx. 
\end{equation}
So this is the term corresponding to the order $\frac{|B_\eps|^2}{2}$. We will see later on that the term (see also Theorem~\ref{thm:expansion_J1_fOmega}, (b) below)
\begin{equation}
    \int_\Dsf 2(u-u_r)d_{x_1}^2 u\;dx
\end{equation}
still appears in the topological expansion, but corresponds to the order $\eps^{-1}\frac{|B_\eps|^2}{2}$, so it is of lower order.

\section{Taylor-like expansion of cost functionals with multiple holes}
Expansions of cost functional with multiple inclusions at once have been studied in several papers before; \cite{a_CALANO_2015a,a_HILANO_2011a}.  Here give a different viewpoint on those expansions.  We apply Theorem~\ref{main:thm_expansion_general} to obtain an expansion with multiple holes for two different cost functionals. 

\subsection{The cost functional $J_1(\Omega)=\int_\Dsf (u_\Omega-u_r)^2\;dx$ with $-\Delta u_\Omega + \gamma u_\Omega = f_\Omega$}
In this section we consider for an open set $\Omega\subset \Dsf$ the cost functional 
\begin{equation}
    J_1(\Omega)  =\int_\Dsf (u_\Omega-u_r)^2\;dx
\end{equation}
subject to $u_\Omega\in H^1(\Omega)$ solves \eqref{eq:state_fOmega}. The target function $u_r$ is assumed to be in $H^1(\Dsf)\cap C(\Dsf)$. 

\subsubsection{Taylor-like expansion of $J_1(\Omega)$ with one hole}

For $d\in \{2,3\}$, $\Omega\subset \Dsf$ open  and  $x_1\in \Dsf\setminus\partial \Omega$ we now compute the expansion in $\eps>0$ of 
\begin{equation}
    J_1(\Omega_\eps(x_1)) = \int_\Dsf (u_\eps-u_r)^2\;dx,
\end{equation}
where $u_\eps = u_{\Omega_\eps(x_1)}$ solves $-\Delta u_\eps +  \gamma u_\eps = f_{\Omega_\eps(x_1)}$ in $\Dsf$ and $u_\eps = g$ on $\Gamma$ and $\partial_n u_\eps = h$ on $\Gamma_N$. Recall that 
\[
    c_{\Omega,x_1}=\sigma_\Omega(x_1)(f_1-f_2).
\]
For later usage, we introduce the adjoint equation:

\begin{definition}
The adjoint $p_\Omega = p\in H^1_\Gamma(\Dsf)$ is defined as the weak  solution to 
\begin{cases2}{eq:adjoint_fOmega}
    -\Delta p_\Omega +\gamma p_\Omega & = -2(u_\Omega-u_r) &&\quad  \text{ in } \Dsf\\
            p_\Omega & = 0 && \quad \text{ on } \Gamma \\
        \partial_n p_\Omega & = 0 && \quad \text{ on } \Gamma_N. 
\end{cases2}
We also introduce the second adjoint $q_{\Omega,x_1} = q\in H^1_\Gamma(\Dsf)$ as solution to 
\begin{cases2}{eq:adjoint_fOmega_q}
    -\Delta q_\Omega + \gamma q_\Omega & = -2d_{x_1}u_\Omega && \quad \text{ in } \Dsf\\
            q_\Omega & = 0 && \quad \text{ on } \Gamma \\
        \partial_n q_\Omega & = 0 && \quad \text{ on } \Gamma_N. 
\end{cases2}
\end{definition}
We note that $q_{\Omega,x_1}$ is smooth in the interior of $\Dsf$ away from the point $x_1$ as $d_{x_1}u_\Omega$ is smooth in $\Dsf$ except at $x_1$.

\begin{theorem}\label{thm:expansion_J1_fOmega}
Let $\Omega\subset \Dsf$ be an open set and $x_1\in \Dsf\setminus \partial \Omega$. 
\begin{itemize}
    \item[(a)] For $d=2$, we have for $\eps>0$ small:
\begin{align}\label{eq:expansion_J1_fOmega}
    \begin{split}
        J_1(\Omega_\eps(x_1)) -J_1(\Omega) =& |B_\eps| \int_{\Dsf} 2(u_\Omega-u_r) d_{x_1} u_\Omega \;dx \\
                                            &+ \frac{|B_\eps|^2}{2} \int_{\Dsf} 2(d_{x_1}u_\Omega)^2 + 2(u_\Omega-u_r)d_{x_1}^2u_\Omega\;dx   + \frR{1}_{\Omega,x_1}(\eps)
\end{split}
\end{align}
or equivalently in terms of the adjoints $p_\Omega,q_{\Omega,x_1}$:
\begin{align}\label{eq:expansion_J1_fOmega2}
    \begin{split}
        J_1(\Omega_\eps(x_1)) -J_1(\Omega) =& - |B_\eps| c_{\Omega,x_1}p_\Omega (x_1) 
        + \frac{|B_\eps|^2}{2} c_{\Omega,x_1}(-q_{\Omega,x_1}(x_1) - \frac{1}{4\pi}\Delta p_\Omega(x_1) )  \\
                                            &  + \frR{1}_{\Omega,x_1}(\eps).
\end{split}
\end{align}
where the remainder satisfies $|\frR{1}_{\Omega,x_1}(\eps)| = o(|B_\eps|^2)$ uniformly in $\Omega$
\item[(b)]  For $d=3$, we have for $\eps>0$ small:
\begin{align}\label{eq:expansion_J1_fOmega_3d}
    \begin{split}
        J_1(\Omega_\eps(x_1)) -J_1(\Omega) =& |B_\eps| \int_{\Dsf} 2(u_\Omega-u_r) d_{x_1} u_\Omega \;dx \\
                                            & + \frac{\eps^{-1}|B_\eps|^2}{2} \int_{\Dsf} 2(u-u_r)d_{x_1}^2u_\Omega\;dx + \frac{|B_\eps|^2}{2} \int_{\Dsf} 2(d_{x_1}u_\Omega)^2 \;dx   + \frR{1}_{\Omega,x_1}(\eps)
\end{split}
\end{align}
or equivalently in terms of the adjoints $p_\Omega,q_{\Omega,x_1}$:
\begin{align}\label{eq:expansion_J1_fOmega2_3d}
    \begin{split}
        J_1(\Omega_\eps(x_1)) -J_1(\Omega) =& -|B_\eps| c_{\Omega,x_1} p_\Omega (x_1)   
        + \frac{|B_\eps|^2}{2} c_{\Omega,x_1}(-q_{\Omega,x_1}(x_1) - \eps^{-1}\frac{3}{20\pi}\Delta p_\Omega(x_1) )  \\
                                            &  + \frR{1}_{\Omega,x_1}(\eps).
\end{split}
\end{align}
where the remainder satisfies $|\frR{1}_{\Omega,x_1}(\eps)| = o(|B_\eps|^2)$ uniformly in $\Omega$. 
\end{itemize}
\end{theorem}
\begin{remark}
    \begin{itemize}
        \item[(a)] 
We notice that the formula \eqref{eq:expansion_J1_fOmega} is what one would expect for a Taylor expansion. This formula does unfortunately not hold for dimension $d=3$ as can be seen from \eqref{eq:expansion_J1_fOmega2_3d}. In fact, the term 
\begin{equation}
    \int_\Dsf 2(d_{x_1}u_\Omega)^2\;dx
\end{equation}
corresponds to the order $|B_\eps|^2\sim\eps^6$, but
\begin{equation}
  \int_\Dsf 2(u-u_r)d_{x_1}^2u_\Omega\;dx   
\end{equation}
corresponds to the order $\eps^{-1}|B_\eps|^2 \sim \eps^5$. This behaviour of dimension dependent derivative could be already observed in  the initial example $J(\Omega) = \int_\Omega f(x)\;dx$ of this paper, where the term corresponding to $|B_\eps|^2$ in dimension $d=3$ was zero. 
\item[(b)] We observe in \eqref{eq:expansion_J1_fOmega2} and \eqref{eq:expansion_J1_fOmega2_3d} the appearance of the constant $\frac{1}{4\pi}\Delta p_\Omega(x_1)$ for $d=2$ and 
    $\frac{3}{20\pi}\Delta p_\Omega(x_1)$ which also appears in the example $J(\Omega) = \int_\Omega f\;dx$ in Lemma~\ref{lem:example_int_f}. There (see Lemma~\ref{lem:example_int_f}) the constant is $\frac{1}{4\pi}\Delta f(x_1)$ and $\frac{3}{20\pi}\Delta f(x_1)$ corresponding to the order 
    $|B_\eps|^2/2$ for $d=2$ and $\eps^{-1}|B_\eps|^2/2$ for $d=3$. This suggest that this term is a geometric term and therefore one  would expect such a term also for other PDE constraints to appear. 
\end{itemize}
\end{remark}

\begin{proof}
We prove the previous theorem in two steps. The starting point is that we can write for $\eps>0$:
\begin{equation}\label{eq:split_cost_J1}
        J(\Omega_\eps(x_1))-J(\Omega) = \int_\Dsf (u_\eps-u)^2\;dx + \int_\Dsf 2(u-u_r)(u_\eps-u)\;dx. 
\end{equation}
Now combining Lemma~\ref{lem:aux_dJ1} and Lemma~\ref{lem:ueps_uOmega_J1_expansion} the applied to the previous equation proves (a) and (b). 
\end{proof}

\begin{lemma}\label{lem:aux_dJ1}
Let $d\in \{2,3\}$ and $\Omega\subset \Dsf$ be an open set. Fix $x_1\in \Dsf\setminus\partial \Omega$.
    \begin{itemize}
        \item[(a)] 
    We have for  $\eps>0$:
\begin{align}
    \int_\Dsf 2(u-u_r)(u_\eps-u)\;dx =& |B_\eps(x_1)| \int_\Dsf 2(u-u_r)d_{x_1}u_\Omega\;dx \\
                                      & + \eps^{d-2}|B_\eps|^2 \int_{\Dsf} 2(u-u_r) d_{x_1}^2 u_\Omega \;dx + \Ct_{\Omega,x_1}(\eps). 
\end{align}
The remainder satisfies $|\Ct_{\Omega,x_1}(\eps)| = o(\eps^{d+4})$ uniformly in $\Omega$. 
\item[(b)] The expansion in (a) is equivalent to
    \begin{align}\label{eq:expansion_quadratic_int_2ur}
    \begin{split}
        \int_\Dsf 2(u-u_r)(u_\eps-u)\;dx =&  -|B_\eps| c_{\Omega,x_1} p_\Omega(x_1) \\
                                      & +  \frac{|B_\eps|^2}{2} \eps^{2-d} \left(\Delta p_\Omega(x_1)c_{\Omega,x_1} \frac{2}{|B|^2} \frac14 \int_B |x|^2-1\;dx \right) + \Ct_{\Omega,x_1}(\eps),
\end{split}
\end{align}
where 
\begin{equation}\label{eq:int_x2_1}
     \frac14\int_B |x|^2-1\;dx = \begin{cases}
        -\frac{\pi}{8} & \text{ for } d=2 \\
        -\frac{2\pi}{15} & \text{ for } d=3.
    \end{cases}
\end{equation}
\end{itemize}
\end{lemma}
\begin{proof}
Throughout the proof we set $u_\eps := u_{\Omega_\eps(x_1)}$ and $u:=u_\Omega$.

\emph{ad (b):} Starting from \eqref{eq:expansion_quadratic_int_2ur} and using the adjoint equation $\Delta p(x_1) = \gamma u(x_1) + 2(u-u_r)$ we obtain 
\begin{equation}
    \Delta p_\Omega(x_1)c_{\Omega,x_1} \frac{2}{|B|^2} \frac14 \int_B |x|^2-1\;dx = (\gamma p(x_1) + 2(u-u_r)(x_1)) \frac{2}{|B|^2} c_{\Omega,x_1}\frac14 \int_B |x|^2-1\;dx.
\end{equation}
Therefore taking into account Remark~\ref{rem:const_int_x2_1_KOmega} we obtain:
\begin{equation}\label{eq:c_Omega_K_Omega}
     c_{\Omega,x_1} \frac{2}{|B|^2} \frac14 \int_B |x|^2-1\;dx =   \frac{2}{|B|^2}
\begin{cases}
        2\pi K_{\Omega,x_1} & \text{ for } d=2\\
        -4\pi K_{\Omega,x_1} & \text{ for } d=3.
    \end{cases}
\end{equation}
Now using the adjoint \eqref{eq:adjoint_fOmega} and \eqref{eq:frak_d} with $z:= \mathfrak d^2_{x_1} u_\Omega$ for $d=2$:
\begin{align}
    c_{\Omega,x_1} \frac{2}{|B|^2} 2\pi K_{\Omega,x_1}\gamma p(x_1) \stackrel{\eqref{eq:frak_d}}{=} -\int_{\Dsf} \nabla z\cdot \nabla p + \gamma z p\;dx \stackrel{\eqref{eq:adjoint_fOmega}}{=} \int_\Dsf 2(u-u_r) z\;dx
\end{align}
and similarly for $d=3$:
\begin{align}
    -c_{\Omega,x_1} \frac{2}{|B|^2} 4\pi K_{\Omega,x_1}\gamma p(x_1) \stackrel{\eqref{eq:frak_d_3d}}{=} -\int_{\Dsf} \nabla z\cdot \nabla p + \gamma z p\;dx \stackrel{\eqref{eq:adjoint_fOmega}}{=} \int_\Dsf 2(u-u_r) z\;dx.
\end{align}
Finally we note that we can write by definition owing to \eqref{eq:c_Omega_K_Omega}:
\begin{align}
    \begin{split}
        2(u-u_r)(x_1)) \frac{2}{|B|^2} c_{\Omega,x_1}\frac14 \int_B |x|^2-1\;dx   \stackrel{\eqref{eq:c_Omega_K_Omega}}{=} &   \frac{2}{|B|^2}C_d K_{\Omega,x_1}2(u-u_r)(x_1) \\
        =& \frac{2}{|B|^2} C_d K_{\Omega,x_1}\int_\Dsf 2(u-u_r) \delta_{x-x_1} dx 
    \end{split}
\end{align}
with $C_d=2\pi$ for $d=2$ and $C_d = -4\pi$ for $d=3$. Therefore the equivalence  follows recalling that 
\begin{equation}
    d_{x_1}^2u_\Omega = \begin{cases}
        \mathfrak d^2_{x_1} u_{\Omega}(x)  + \frac{2}{|B|^2} 2\pi K_{\Omega,x_1} \delta_{x-x_1} & \text{ for } d=2\\
        \mathfrak d^2_{x_1} u_{\Omega}(x)  - \frac{2}{|B|^2} 4\pi K_{\Omega,x_1} \delta_{x-x_1} & \text{ for } d=3.
    \end{cases}
\end{equation}
This finishes the proof of the equivalence of (a) and (b).

\emph{ad (a):} We can write
    \begin{equation}
        J(\Omega_\eps(x_1))-J(\Omega) = \int_\Dsf (u_\eps-u)^2\;dx + \int_\Dsf 2(u-u_r)(u_\eps-u)\;dx. 
    \end{equation}
We first expand the term $\int_\Dsf 2(u-u_r)(u_\eps-u)\;dx$.  The difference $u_\eps - u$ satisfies with the adjoint variable $p$:
\begin{equation}
\int_\Dsf\nabla (u_\eps-u)\cdot \nabla p + \gamma (u_\eps-u)p \;dx= c_{\Omega,x_1} \int_{B_\eps(x_1)}p\;dx
\end{equation}
and testing the adjoint equation \eqref{eq:adjoint_fOmega} with $u_\eps -u$ gives:
\begin{equation}
\int_{\Dsf}\nabla p\cdot \nabla (u_\eps-u) + p(u_\eps - u)\;dx =-\int_\Dsf 2(u-u_r)(u-u_\eps)\;dx.
\end{equation}
Therefore
\begin{equation}
    \int_\Dsf 2(u-u_r)(u_\eps-u)\;dx = -c_{\Omega,x_1}\int_{B_\eps(x_1)} p\;dx.
\end{equation}
We recall $c_{\Omega,x_1} = \sigma_\Omega(x_1) (f_1-f_2)$. 
Now since 
\begin{equation}
    \int_{B_\eps(x_1)} p\;dx = \eps^d\int_B p(x_1+\eps x)\;dx
\end{equation}
a Taylor expansion in $\eps=0$ and noting that $\int_B \nabla^kp(x_1)[x]^k\;dx=0$ for $k$ odd (see Remark~\ref{rem:int_nabla_f_xk}), yields
\begin{align}
    \int_{B_\eps(x_1)} p\;dx =& |B_\eps(x_1)|p(x_1) + \eps^{d+2} \frac12 \int_B \nabla^2 p(x_1)x\cdot x\;dx \\
                              & +  \underbrace{\eps^{d+4}\frac{1}{4!}\int_0^1\int_B \nabla^4 p(x_1+sx)[x]^4\;dx\;ds}_{=:\Ct_{\Omega,x_1}(\eps)}.
\end{align}
It is readily checked that $\Ct_{\Omega,x_1}(\eps) = O(\eps^{d+4})=o(|B_\eps|^2)$ uniformly in $\Omega$. 
Since $x = \frac12 \nabla (|x|^2-1)$ and $\Div(\nabla^2 p(x_1)x)=\Delta p(x_1)$, we have
\begin{align}\label{eq:partial_int_nablap}
    \begin{split}
    -\frac12\int_B\nabla^2 p(x_1)x\cdot x \;dx &=  -\frac14\int_B \nabla^2 p(x_1)x\cdot \nabla(|x|^2-1)\;dx \\
                                        & = \frac14 \Delta p(x_1)\int_B (|x|^2-1)  \;dx.
\end{split}
\end{align}
Combining the previous equations yields \eqref{eq:expansion_quadratic_int_2ur}. The equation \eqref{eq:int_x2_1} is readily checked using spherical/polar coordinates.

\end{proof}

Next we turn our attention to the first term on the right hand side of \eqref{eq:split_cost_J1}.

\begin{lemma}\label{lem:ueps_uOmega_J1_expansion}
    Let $d\in \{2,3\}$, $\Omega\subset \Dsf$ be an open set and $x_1\in \Dsf\setminus\partial \Omega$.
    \begin{itemize}
        \item[(a)] We have the expansion
            \begin{equation}\label{eq:expansion_J1_square_term}
                \int_\Dsf (u_{\Omega_\eps(x_1)} - u_\Omega)^2\;dx = |B_\eps|^2\int_\Dsf (d_{x_1}u_\Omega)^2\;dx + \Ct_{\Omega}(\eps),
\end{equation}
where $|\Ct_\Omega(\eps)|= o(|B_\eps|^2)$ uniformly in $\Omega$. 
\item[(b)] The expansion in (a) is equivalent to 
            \begin{equation}\label{eq:expansion_J1_square_term2}
                \int_\Dsf (u_{\Omega_\eps(x_1)} - u_\Omega)^2\;dx = -\frac{|B_\eps|^2}{2} c_{\Omega,x_1} q_{\Omega,x_1}(x_1) + \Ct_{\Omega}(\eps).
\end{equation}
\end{itemize}
\end{lemma}
\begin{proof}
    Throughout the proof we set $u_\eps := u_{\Omega_\eps(x_1)}$ and $u:=u_\Omega$.

    \emph{ad (b):}  The equivalence follows by testing the weak formulation of $q_{\Omega,x_1}$ with $d_{x_1}u_\Omega$ and testing the weak formulation of $d_{x_1}u_\Omega$ with $q_{\Omega,x_1}$:
    \begin{align}
    \int_{\Dsf} 2(d_{x_1}u_\Omega)^2\;ds& \stackrel{\eqref{eq:adjoint_fOmega_q}}{=} \int_{\Dsf} \nabla (q_{\Omega,x_1}) \cdot \nabla(d_{x_1}u_\Omega) + \gamma q_{\Omega,x_1} d_{x_1}u_\Omega \;dx 
    \stackrel{\eqref{eq:ueps_state_derivative_fomega_equation}}{=}  c_{\Omega,x_1} q_{\Omega,x_1}(x_1). 
    \end{align}
    This finishes the proof of  (b).

    \emph{ad (a):} We first split the integral:
    \begin{equation}\label{eq:split_cost_J1_proof}
\int_\Dsf (u_\eps-u)^2\;dx = \int_{\Dsf\setminus B_\eps(x_1)}(u_\eps-u)^2\;dx + \int_{B_\eps(x_1)} (u_\eps-u)^2\;dx.
\end{equation}
We recall that for $d=2$ in $\Dsf$:
\begin{equation}\label{eq:expansion_ueps_u_proof_split}
    u_\eps = u + \eps^2 (\Uk{1}\circ T_\eps^{-1} + \vk{1} + \Rki{1}{1}(\eps)) + \eps^4 (\Uk{2}\circ T_\eps^{-1} + \vk{2} + \Rki{2}{1}(\eps)) + \Ce_\Omega(\eps)
\end{equation}
and for $d=3$:
\begin{equation}\label{eq:expansion_ueps_u_proof_split_3d}
    u_\eps = u + \eps^3 (\Uk{1}\circ T_\eps^{-1} + \vk{1}) + \eps^6 ( \eps^{-1}\Uk{2}\circ T_\eps^{-1} + \vk{2} ) + \Ce_\Omega(\eps).
\end{equation}
We will proceed only for the case $d=2$, but $d=3$ works the same way with obvious adaptations. Therefore we have for $d=2$ on the unit ball $B$
\begin{equation}
    (u_\eps-u)\circ T_\eps = \eps^2\underbrace{ (\Uk{1} + \vk{1}\circ T_\eps + \Rki{1}{1}(\eps))}_{:=\Zk{1}_\eps} + \eps^4\underbrace{(\Uk{2} + \vk{2}\circ T_\eps + \Rki{2}{1}(\eps) + \frac{\Ce_\Omega(\eps)\circ T_\eps}{\eps^4})}_{=:\Zk{2}_\eps}.
\end{equation}
Recall that $\Ce_\Omega(\eps)\circ T_\eps/\eps^4 = \Uk{3}_\eps$ (see Corollary~\ref{cor:expansion_remainder_D}). 
So we can write the integral over $B_\eps(x_1)$ with the change of variables $T_\eps(x)=x_1+\eps x$ as follows
\begin{align}
    \int_{B_\eps} (u_\eps-u)^2\;dx = \eps^2 \int_B \eps^4(\Zk{1}_\eps)^2 + \eps^8(\Zk{2}_\eps)^2 + \eps^6 2\Zk{1}_\eps \Zk{2}_\eps \;dx =: \frT{1}_\Omega(\eps). 
\end{align}
Since $\Uk{1},\Uk{2}$ and $\vk{1},\vk{2}$ are continuous on $B$ the only term that is unbounded in $\Zk{1}_\eps$ is $\Rki{1}{1}(\eps) = c_1 \ln(\eps)$ and in $\Zk{2}_\eps$ it is $\Rki{2}{1}(\eps) = c_2\ln(\eps)$ with two constants $c_1,c_2$. Therefore we can bound the single terms as follows:
\begin{align}
    \begin{split}
    \eps^6 \int_B (\Zk{1}_\eps)^2 \;dx & \le  \eps^6(C_1(\ln(\eps))^2 + C_2)  = O(\eps^5)  \\
    \eps^{10} \int_B (\Zk{2}_\eps)^2\;dx & \le \eps^8(C_1(\ln(\eps))^2 + C_2) = O(\eps^7) \\
    \eps^8 \left|\int_B 2\Zk{1}_\eps \Zk{2}_\eps \;dx\right| & \le  \eps^8 2\| \Zk{1}_\eps\|_{L^2(B)}\|\Zk{2}_\eps\|_{L^2(B)} \le \eps^8(C_1(\ln(\eps))^2 + C_2) = O(\eps^7). 
\end{split}
\end{align}
Here we also used that $\Uk{3}_\eps = \frac{\Ce_\Omega(\eps)\circ T_\eps}{\eps^4}$ is bounded by $C \sqrt{-\ln(\eps)}$ in $H^1(B)$  according to Corollary~\ref{cor:expansion_remainder_D} and Theorem~\ref{thm:remainder_estimate_Deps}. This shows that $\frT{1}_\Omega(\eps) = o(|B_\eps|^2)$ uniformly in $\Omega$ for $d=2$.

Now we turn our attention to the integral in \eqref{eq:split_cost_J1_proof} over the set $\Dsf\setminus B_\eps(x_1)$. For this we note that  $\mathfrak d^2_{x_1}u_\Omega = \frac{2}{|B|^2}(\Rki{2}{1}(x-x_1) + \vk{2})$ and thus \eqref{eq:expansion_ueps_u_proof_split}  becomes after rearranging:
\begin{align}\label{eq:ueps_uomega_Deps}
    \begin{split}
        u_\eps - u =&  \eps^2 |B|d_{x_1}u + \eps^4 \frac{|B|^2}{2}  \mathfrak d_{x_1}^2u \\
                    & + \eps^4 \frac{|B|^2}{2} (\Uk{2}\circ T_\eps^{-1} - \Rki{2}{1}(x-x_1) + \Rki{2}{1}(\eps)) + \Ce_\Omega(\eps) \quad \text{ on } \Dsf\setminus B_\eps(x_1).
\end{split}
\end{align}
We can further write:
\begin{align}
    \begin{split}
    \int_{\Dsf\setminus B_\eps(x_1)}(u_\eps - u)^2\;dx =&     \int_{\Dsf\setminus B_\eps(x_1)} \eps^4|B|^4 (d_xu)^2\;dx + \Ct_\Omega(\eps) \\
    =&  \int_{\Dsf} \eps^4|B|^4 (d_xu)^2\;dx \underbrace{- \int_{B_\eps(x_1)} \eps^4|B|^4 (d_xu)^2\;dx + \Ct_\Omega(\eps)}_{=:\frT{2}_\Omega(\eps)}
\end{split}
\end{align}
with a remainder term $\Ct_\Omega(\eps)$. We will show that $|\frT{2}_\Omega(\eps)|=o(\eps^4)$ uniformly in $\Omega$ for $d=2$.

For this we note that the remainder $\Ct_\Omega(\eps)$ consists of integrals over $\Dsf\setminus B_\eps(x_1)$ of products of all combinations (including itself and up to constants) of 
\begin{equation}
\eps^2d_{x_1}u, \quad  \eps^4 \mathfrak d_{x_1}^2u, \quad \eps^4 (\Uk{2}\circ T_\eps^{-1} - \Rki{2}{1}(x-x_1) + \Rki{2}{1}(\eps)), \quad \Ce_\Omega(\eps)
\end{equation}
except the term $\int_{\Dsf\setminus B_\eps(x_1)} \eps^4|B|^4 (d_xu)^2\;dx$. Each of these products has at least a power of $\eps^6$.  Therefore it is sufficient that the integral over $\Dsf\setminus B_\eps(x_1)$ of the square of each of those terms is bounded.  The products of different functions can be reduced to this case by H\"older's inequality.  Now by the Sobolev embedding theorem since $\mathfrak d_{x_1}^2u_\Omega, d_{x_1}^2u_\Omega\in W^{1,p}_\Gamma(\Dsf)$, we  have
\begin{equation}
    d_{x_1}u_\Omega,\; \mathfrak d_{x_1}^2u_\Omega \in \begin{cases}
        L^p(\Dsf) & \text{ for } d=2 \;\text{ and }\; 1\le p <\infty \\
        L^p(\Dsf) & \text{ for } d=3\;\text{ and }\; 1\le p < 3.
    \end{cases}
\end{equation}
Therefore we have for $r,s >0$, $\frac1r + \frac1s=1$ and $s>1$ small enough:
\begin{equation}
   \eps^4 \int_{B_\eps(x_1)} (d_{x_1}u)^2\;dx\le \eps^4 |B_\eps|^{\frac{1}{r}} \|(d_{x_1}u)^2\|_{L^s(\Dsf)}  \le C \eps^4 \eps^{\frac1r} = o(\eps^4)
\end{equation}
and also 
\begin{equation}
    \eps^8\int_{\Dsf\setminus B_\eps(x_1)}  (\mathfrak d_{x_1}^2u)^2 \;dx \le \eps^8 \|\mathfrak d_{x_1}^2u\|_{W^{1,p}(\Dsf)} = O(\eps^8).
\end{equation}
Further using Lemma~\ref{lem:bound_rk2_Uk2}
\begin{align}
    \begin{split}
    \eps^8 \int_{\Dsf\setminus B_\eps} (\Uk{2}\circ T_\eps^{-1}& - \Rki{2}{1}(x-x_1) + \Rki{2}{1}(\eps))^2\;dx \\
                                                               & = \eps^6 \int_{\Dsf_\eps^{-1}\setminus B} (\Uk{2} - \Rki{2}{1})^2\;dx \le \eps^6(C_1 + C_2)|\ln(\eps)|) = o(\eps^5).
\end{split}
\end{align}
As for the remainder we have by Corollary~\ref{cor:expansion_remainder_D}:
\begin{equation}
    \|\Ce_\Omega(\eps)\|_{L^2(\Dsf\setminus B_\eps(x_1))}^2 = \begin{cases}
        O(\eps^{10}(-\ln(\eps))) & \text{ for } d=2 \\
        O(\eps^{14}) & \text{ for } d=3
    \end{cases} 
\end{equation}
uniformly in $\Omega$.  Combining the previous results we see that $\Ct_\Omega(\eps) := \frT{1}_\Omega(\eps) + \frT{2}_\Omega(\eps)$ satisfies $|\Ct_\Omega(\eps)| = o(\eps^4)$ uniformly in $\Omega$ for $d=2$.

\end{proof}

\subsubsection{Taylor-like expansion of $J_1(\Omega)$ with multiple holes}
We now apply Theorem~\ref{main:thm_expansion_general} to obtain the following result for $d=2$.

\begin{theorem}\label{main:expansion_J1_multiple_2d}
Assume $d=2$. Let $\Omega\subset \Dsf$ be an open set. For $x_1,\ldots,x_n\in \Dsf\setminus\partial \Omega$, $n\ge 1$ with $x_i\ne x_j$ for $i\ne j$, we have
\begin{align}
    J(\Omega_{\eps_1,\ldots,\eps_n})  = J(\Omega) + D\VJ (\Omega)(x_1,\ldots, x_n)\cdot \begin{pmatrix}
|B_{\eps_1}|\\ \vdots \\ |B_{\eps_n}|
\end{pmatrix} 
        + & \frac12D^2 \VJ(\Omega)(x_1,\ldots, x_n) \begin{pmatrix}
|B_{\eps_1}| \\\vdots \\ |B_{\eps_n}|
\end{pmatrix}\cdot \begin{pmatrix}
|B_{\eps_1}| \\ \vdots \\|B_{\eps_n}|
\end{pmatrix}  
                              \\
          &+ o(|B_{\eps_1}|^2+\cdots + |B_{\eps_n}|^2),
\end{align}
where for $i,j\in \{1,\ldots, n\}$:
\begin{align}
    (D^2 \VJ(\Omega)(x_1,\ldots, x_n))_{ii} & = \int_\Dsf 2 (d_{x_i}u_\Omega)^2\;dx + \int_\Dsf 2(u-u_r)d_{x_1}^2 u_\Omega \;dx \\
    (D^2 \VJ(\Omega)(x_1,\ldots, x_n))_{ij} & = \int_\Dsf  2(d_{x_i}u_\Omega) (d_{x_j} u_\Omega)\;dx, \quad i\ne j,
\end{align}
and for $i\in \{1,\ldots, n\}$:
\begin{equation}
    (D\VJ (\Omega)(x_1,\ldots, x_n))_i :=  \int_\Dsf 2(u-u_r)(d_{x_i}u_\Omega)\;dx.
\end{equation}
\end{theorem}
\begin{proof}
 We apply Theorem~\ref{main:thm_expansion_general} with $\ell(\eps):=0$,  $d=2$ and $n\ge 1$ arbitrary. 

  Assumption (A1) is satisfied in view of Theorem~\ref{thm:expansion_J1_fOmega}.   To check Assumption (A2) let $x_1,x_2\in\Dsf\setminus\partial \Omega$, $x_1\ne x_2$, then 
  \begin{equation}
      DJ(\Omega_{\eps_2}(x_2))(x_1) = \int_\Dsf 2(u_{\Omega_{\eps_2}(x_2)} - u_r) d_{x_1}u_\Omega\;dx. 
  \end{equation}
  Here we used that $d_{x_1} u_{\Omega_{\eps_2}(x_2)}= d_{x_1} u_{\Omega}$ for $\eps_2>0$ small. Now notice that using \eqref{eq:state_fOmega} we see that $u_{\Omega_{\eps_2}(x_2)} - u_\Omega$ satisfies:
  \begin{equation}\label{eq:diff_eps2}
    \int_\Dsf \nabla ( u_{\Omega_{\eps_2}(x_2)} - u_\Omega )\cdot \nabla \varphi + \gamma ( u_{\Omega_{\eps_2}(x_2)} - u_\Omega ) \varphi \;dx = c_{\Omega,x_2} \int_{B_{\eps_2}(x_2)}\varphi\;dx  
\end{equation}
for $\varphi\in H^1_\Gamma(\Dsf)$.
  Therefore using the  adjoint $q_{\Omega,x_1} = q$ we obtain
  \begin{align}\label{eq:diff_DJ(Omega_eps2)}
      \begin{split}
      DJ(\Omega_{\eps_2}(x_2))(x_1)-DJ(\Omega)(x_1)  & = \int_\Dsf 2(u_{\Omega_{\eps_2}(x_2)} - u_\Omega ) d_{x_1}u_\Omega\;dx \\
                                                     &\stackrel{\eqref{eq:adjoint_fOmega_q}}{=} -\int_\Dsf \nabla q_{\Omega,x_1} \cdot \nabla (u_{\Omega_{\eps_2}(x_2)} - u_\Omega ) + \gamma q_{\Omega,x_1} (u_{\Omega_{\eps_2}(x_2)} - u_\Omega )\;dx \\
                                                     & \stackrel{\eqref{eq:diff_eps2}}{=} -c_{\Omega,x_2} \int_{B_\eps(x_2)} q_{\Omega,x_1}\;dx \\
                                                     & = -c_{\Omega,x_2}\eps^2 \int_B q_{\Omega,x_1}(x_2+\eps x) \;dx.
  \end{split}
  \end{align}
  Now since $-\Delta q_{\Omega,x_1} = -\gamma q_{\Omega,x_1} + 2 d_{x_1}u_\Omega$ and since $d_{x_1}u_\Omega$ is smooth near $x_2$ it follows that $q_{\Omega,x_1}$ is smooth near $x_2$ and depends continuously on $\Omega$. Therefore we can use a Taylor expansion to obtain:
\begin{equation}
    DJ(\Omega_{\eps_2}(x_2))(x_1)-DJ(\Omega)(x_1) = -c_{\Omega,x_2} |B_{\eps_2}|q_{\Omega,x_1}(x_2) + \underbrace{(\eps_2)^4\int_0^1\int_B  \frac12 \nabla^2 q_{\Omega,x_1}(x_2 + sx )x\cdot x\;dx\,ds}_{=:\frR{2}_{\Omega,x_1,x_2}(\eps)}.
\end{equation}
Finally noting that $D^2J(\Omega)(x_1,x_2) = \int_\Dsf (d_{x_1}u_\Omega)(d_{x_2}u_\Omega)\;dx =  -c_{\Omega,x_1} q_{\Omega,x_1}(x_2)$ we see that Assumption~(A2) is satisfied with $|\frR{2}_{\Omega,x_1,x_2}(\eps)| = O(\eps^4)$ uniformly. 

Assumption~(A3) is satisfied since $\ell(\eps)=0$.  Finally to check Assumption~(A4) note that for $x_1\ne x_2$ and $x_1,x_2\ne x_3$:
\begin{equation}
    D^2J(\Omega_{\eps_3})(x_1,x_2) = \int_{\Dsf} 2 (d_{x_1}u_{\Omega_{\eps_3}})(d_{x_2}u_{\Omega_{\eps_3}}) \;dx. 
\end{equation}
Now $d_{x_i}u_{\Omega_{\eps_3}} = d_{x_i}u_{\Omega}$, $i=1,2$, for $\eps_3>0$ small, which shows (A4) for $x_1\ne x_2$. In case $x_1=x_2$ and $x_1,x_2\ne x_3$ we have according to Theorem~\ref{thm:expansion_J1_fOmega}:
\begin{equation}
    D^2J(\Omega_{\eps_3})(x_1,x_1)  = \int_{\Dsf} 2(d_{x_1}u_{\Omega_{\eps_1}})^2 \;dx -c_{\Omega,x_1}(q_{\Omega_{\eps_2},x_1}(x_1) - \frac{1}{4\pi} (\gamma p_{\Omega_{\eps_3}}(x_1) + 2(u_{\Omega_{\eps_3}} - u_r)(x_1)). 
\end{equation}
The first term on the right hand side is treated as in the case $x_1\ne x_2$. As for the second term we $q_{\Omega_{\eps_3},x_1}=q_{\Omega,x_1}$ since $d_{x_1} u_{\Omega_{\eps_3}} = d_{x_1} u_{\Omega}$. So for the terms involving $p_{\Omega_{\eps_3}}(x_1)$ and $u_{\Omega_{\eps_3}}(x_1)$ we note that by elliptic regularity for $\delta >0$ small, $d=2$ and $p>2$:
\begin{align}\label{eq:elliptic_regularity}
    \begin{split}
    |p_{\Omega_{\eps_3}}(x_1) -  p_\Omega(x_1)| & \le \|p_{\Omega_{\eps_3}} -  p_\Omega\|_{W^{1,p}(B_{\delta}(x_3))}  
                                                 \le \|u_{\Omega_{\eps_3}} - u_\Omega\|_{L^p(B_{2\delta}(x_3))}\\
    |u_{\Omega_{\eps_3}}(x_1) -  u_\Omega(x_1)|  & \le \|u_{\Omega_{\eps_3}} -  u_\Omega\|_{W^{1,p}(B_{\delta}(x_3))} 
                                                  \le \|f_{\Omega_{\eps_3}}-f_\Omega\|_{L^p(B_{2\delta}(x_3))}. 
\end{split}
\end{align}
So we see that
\begin{align}
|\frR{4}_{\Omega,x_1,x_2}(\eps_3)|=&  |D^2J(\Omega_{\eps_3})(x_1,x_1)  -D^2J(\Omega)(x_1,x_1)| 
    \le C\|f_{\Omega_{\eps_3}} - f_\Omega\|_{L^p(B_{2\delta}(x_3)} 
     \le  C|B_{\eps_3}|^{\frac{1}{p}},
\end{align}
which is Assumption~(A4) for $x_1=x_2$.

\end{proof}

\begin{theorem}\label{main:expansion_J1_multiple_3d}
Let $\Omega\subset \Dsf$ be an open set. For $x_1,\ldots,x_n\in \Dsf\setminus\partial \Omega$, $n\ge 1$ with $x_i\ne x_j$ for $i\ne j$, we have
\begin{align}
    \begin{split}
    J(\Omega_{\eps_1,\ldots,\eps_n})  =  J(\Omega) + D\VJ (\Omega)(x_1,\ldots, x_n)\cdot \begin{pmatrix}
|B_{\eps_1}|\\ \vdots \\ |B_{\eps_n}|
\end{pmatrix} 
        &+ 
        \frac12D^{\frac32} \VJ(\Omega)(x_1,\ldots, x_n) \begin{pmatrix}
            \eps^{-\frac12}|B_{\eps_1}| \\\vdots \\ \eps^{-\frac12}|B_{\eps_n}|
\end{pmatrix}\cdot \begin{pmatrix}
\eps^{-\frac12}|B_{\eps_1}| \\ \vdots \\  \eps^{-\frac12}|B_{\eps_n}|
\end{pmatrix}  \\
          &+ \frac12D^2 \VJ(\Omega)(x_1,\ldots, x_n) \begin{pmatrix}
|B_{\eps_1}| \\\vdots \\ |B_{\eps_n}|
\end{pmatrix}\cdot \begin{pmatrix}
|B_{\eps_1}| \\ \vdots \\|B_{\eps_n}|
\end{pmatrix}     \\
          &+ o(|B_{\eps_1}|^2+\cdots + |B_{\eps_n}|^2),
    \end{split}
\end{align}
where for $i,j=1,\ldots, n$:
\begin{align}
    (D^{\frac32} \VJ(\Omega)(x_1,\ldots, x_n))_{ii} & = \int_\Dsf 2(u_\Omega-u_r)d_{x_1}^2 u_\Omega \;dx \\
    (D^{\frac32} \VJ(\Omega)(x_1,\ldots, x_n))_{ij} & = 0, \quad i\ne j \\
    (D^2 \VJ(\Omega)(x_1,\ldots, x_n))_{ii} & = \int_\Dsf 2(d_{x_i}u_\Omega)^2\;dx \\
    (D^2 \VJ(\Omega)(x_1,\ldots, x_n))_{ij} & = \int_\Dsf  2(d_{x_i}u_\Omega) (d_{x_j} u_\Omega)\;dx 
\end{align}
and for $i\in \{1,\ldots,n\}$:
\begin{equation}
    (D\VJ (\Omega)(x_1,\ldots, x_n))_i =  \int_\Dsf 2(u_\Omega-u_r)(d_{x_i}u_\Omega)\;dx.
\end{equation}
\end{theorem}
\begin{proof}
    We apply Theorem~\ref{main:thm_expansion_general} with $\ell(\eps):= \frac12\eps^{-1}|B_\eps|^2$, $d=3$ and $n\ge 1$ arbitrary. Since the proof is similar to the one of Theorem~\ref{main:expansion_J1_multiple_3d} it is omitted.
\end{proof}

\subsection{The cost functional $J_2(\Omega)=\int_{\Gamma_N} (u_\Omega-u_r)^2\;ds$ with $-\Delta u_\Omega + \gamma u_\Omega = f_\Omega$}
In this section we consider for $\Omega\subset \Dsf$ the cost functional
\begin{equation}
    J_2(\Omega) =\int_{\Gamma_N} (u_\Omega-u_r)^2\;ds
\end{equation}
subject to $u_\Omega\in H^1_\Gamma(\Omega)$ solves \eqref{eq:state_fOmega}. The target function $u_r$ is assumed to be in $H^1(\Gamma_N)$. 

\subsubsection{Taylor-like expansion of $J_2(\Omega)$}
Let $\Omega\subset \Dsf$ be an open set. We now compute the expansion for $x_1\in \Dsf\setminus\partial \Omega$ and  $\eps>0$ of
\begin{equation}
    J_2(\Omega_\eps(x_1)) = \int_{\Gamma_N} (u_\eps-u_r)^2\;ds,
\end{equation}
where $u_\eps = u_{\Omega_\eps(x_1)}$ solves $-\Delta u_\eps + \gamma u_\Omega = f_{\Omega_\eps(x_1)}$ in $\Dsf$ and $u_\eps = g$ on $\Gamma$ and $\partial_n u_\eps = h$ on $\Gamma_N$.

We introduce again two adjoint variables associated with $J_2$ and $d_{x_1}u_\Omega$, respectively. 
\begin{definition}
The adjoint $p_\Omega = p\in H^1_\Gamma(\Dsf)$ is defined as the weak  solution to 
\begin{cases2}{eq:adjoint_fOmega_boundary}
    -\Delta p_\Omega +\gamma p_\Omega & = 0  && \quad \text{ in } \Dsf\\
            p_\Omega & = 0 && \quad \text{ on } \Gamma \\
        \partial_n p_\Omega & = -2(u_\Omega-u_r) && \quad \text{ on } \Gamma_N. 
\end{cases2}
We also introduce the second adjoint $q_{\Omega,x_1} = q\in H^1_\Gamma(\Dsf)$ as solution to 
\begin{cases2}{eq:adjoint_fOmega_q_boundary}
    -\Delta q_\Omega + \gamma q_\Omega & = 0 && \quad \text{ in } \Dsf\\
            q_\Omega & = 0 && \quad \text{ on } \Gamma \\
        \partial_n q_\Omega & = -2d_{x_1}u_\Omega && \quad \text{ on } \Gamma_N. 
\end{cases2}
\end{definition}
Note that $q_{\Omega,x_1}$ is smooth in the interior of $\Dsf$ by a bootstrapping argument.  We now state the main result of this section.
\begin{theorem}\label{thm:expansion_J2_fOmega}
Let $\Omega\subset \Dsf$ be an open set and $x_1\in \Dsf\setminus \partial \Omega$. 
\begin{itemize}
    \item[(a)] For $d=2$, we have for $\eps>0$ small:
\begin{align}\label{eq:expansion_J2_fOmega}
    \begin{split}
        J_1(\Omega_\eps(x_1)) -J_1(\Omega) =& |B_\eps| \int_{\Gamma_N} 2(u_\Omega-u_r) d_{x_1} u_\Omega \;ds \\
                                            &+ \frac{|B_\eps|^2}{2} \int_{\Gamma_N} 2(d_{x_1}u_\Omega)^2 + 2(u-u_r)\mathfrak d_{x_1}^2u_\Omega\;ds   + \Ct_\Omega(\eps)
\end{split}
\end{align}
or equivalently in terms of the adjoints $p_\Omega,q_{\Omega,x_1}$:
\begin{align}\label{eq:expansion_J2_fOmega2}
    \begin{split}
        J_1(\Omega_\eps(x_1)) -J_1(\Omega) =& -|B_\eps| c_{\Omega,x_1} p_\Omega (x_1)   
        + \frac{|B_\eps|^2}{2} c_{\Omega,x_1}(-q_{\Omega,x_1}(x_1) - \frac{1}{4\pi}\Delta p_\Omega(x_1) )  \\
                                            &  + \Ct_\Omega(\eps),
\end{split}
\end{align}
where the remainder satisfies $|\Ct_\Omega(\eps)| = o(|B_\eps|^2)$ uniformly in $\Omega$. 
\item[(b)]  For $d=3$, we have for $\eps>0$ small:
\begin{align}\label{eq:expansion_J2_fOmega_3d}
    \begin{split}
        J_1(\Omega_\eps(x_1)) -J_1(\Omega) =& |B_\eps| \int_{\Gamma_N} 2(u_\Omega-u_r) d_{x_1} u_\Omega \;ds \\
                                            & + \frac{\eps^{-1}|B_\eps|^2}{2} \int_{\Gamma_N} 2(u_\Omega-u_r)\mathfrak d_{x_1}^2u_\Omega\;ds + \frac{|B_\eps|^2}{2} \int_{\Gamma_N} 2(d_{x_1}u_\Omega)^2 \;ds   + \Ct_\Omega(\eps)
\end{split}
\end{align}
or equivalently in terms of the adjoints $p_\Omega,q_{\Omega,x_1}$:
\begin{align}\label{eq:expansion_J2_fOmega2_3d}
    \begin{split}
        J_1(\Omega_\eps(x_1)) -J_1(\Omega) =& -|B_\eps| c_{\Omega,x_1} p_\Omega (x_1)   
        + \frac{|B_\eps|^2}{2} c_{\Omega,x_1}(-q_{\Omega,x_1}(x_1) - \eps^{-1}\frac{3}{20\pi}\Delta p_\Omega(x_1) )  \\
                                            &  + \Ct_\Omega(\eps),
\end{split}
\end{align}
where the remainder satisfies $|\Ct_\Omega(\eps)| = o(|B_\eps|^2)$ uniformly in $\Omega$.  
\end{itemize}
\end{theorem}

The proof is again based on the splitting: (here $u:=u_\Omega$)
\begin{equation}\label{eq:split_cost_J2}
    J(\Omega_\eps(x_1))-J(\Omega) = \int_{\Gamma_N} (u_\eps-u)^2\;ds + \int_{\Gamma_N} 2(u-u_r)(u_\eps-u)\;ds. 
\end{equation}

Then Theorem~\ref{thm:expansion_J2_fOmega} follows immediately from the following two lemmata. 

\begin{lemma}\label{lem:aux_dJ2}
    Let $d\in \{2,3\}$, $\Omega\subset \Dsf$ be an open set and $x_1\in \Dsf\setminus \partial \Omega$.
    \begin{itemize}
        \item[(a)] 
    We have for $\eps>0$:
    \begin{align}\label{eq:expansion_quadratic_int_2ur_boundary}
    \begin{split}
        \int_{\Gamma_N} 2(u-u_r)(u_\eps-u)\;ds =&  -|B_\eps|c_{\Omega,x_1} p_\Omega(x_1) \\
                                      & +  \frac{|B_\eps|^2}{2} \eps^{2-d} \left(\Delta p_\Omega(x_1)c_{\Omega,x_1} \frac{2}{|B|^2} \frac14 \int_B |x|^2-1\;ds \right) + \Ct_{\Omega,x_1}(\eps),
\end{split}
\end{align}
where 
\begin{equation}\label{eq:int_x2_1_boundary}
     \frac14\int_B |x|^2-1\;ds = \begin{cases}
        -\frac{\pi}{8} & \quad \text{ for } d=2 \\
        -\frac{2\pi}{15} & \quad \text{ for } d=3.
    \end{cases}
\end{equation}
The remainder satisfies $|\Ct_{\Omega,x_1}(\eps)| = o(\eps^{d+4})$ uniformly in $\Omega$.
\item[(b)] The expansion in (a) is equivalent to 
\begin{align}
    \int_{\Gamma_N} 2(u-u_r)(u_\eps-u)\;ds =& |B_\eps| \int_{\Gamma_N} 2(u-u_r)d_{x_1}u_\Omega\;ds \\
                                      & + \eps^{d-2}|B_\eps|^2 \int_{\Gamma_N} 2(u-u_r) \mathfrak d_{x_1}^2 u_\Omega \;ds + \Ct_{\Omega,x_1}(\eps). 
\end{align}
\end{itemize}
\end{lemma}
\begin{proof}
    The proof of (a) is identical with the one of (a) of Lemma~\ref{lem:aux_dJ1}: Notice that $u_\eps-u_0$ satisfies:
    \begin{equation}
        \int_{\Gamma_N} 2(u-u_r)(u_\eps-u)\;ds \stackrel{\eqref{eq:adjoint_fOmega_boundary}}{=}     -\int_{\Dsf}\nabla (u_\eps-u_0) \cdot \nabla p + \gamma (u_\eps-u)p \;dx \stackrel{\eqref{eq:state_fOmega}}{=}   -c_{\Omega,x_1}\int_{B_\eps(x_1)} p\;dx.
    \end{equation}
    Now the expansion of the right hand side follows the same lines as the proof of Lemma~\ref{lem:aux_dJ2}. Notice that $\mathfrak d_{x_1}^2u_\Omega$ appears instead of $d_{x_1}^2u_\Omega $, which is due to the fact that $\Delta p(x_1) = \gamma u_\Omega(x_1)$ compared to $\Delta p(x_1) = \gamma u_\Omega(x_1) + 2(u_\Omega-u_r)(x_1)$ in case of the cost functional $J_1(\Omega)$. 

    The proof of the equivalence of (a) and (b) is also the same as the one of Lemma~\ref{lem:aux_dJ1} with the obvious changes.

\end{proof}

\begin{lemma}
   Let $d\in \{2,3\}$, $\Omega\subset \Dsf$ be an open set and $x_1\in \Dsf\setminus \partial \Omega$.
\begin{itemize} 
        \item[(a)] We have the expansion
    \begin{equation}\label{eq:expansion_J3_ueps_u}
        \int_{\Gamma_N} (u_{\Omega_\eps(x_1)} - u_\Omega)^2\;ds = |B_\eps|^2\int_{\Gamma_N} (d_{x_1}u_\Omega)^2\;ds + \Ct_{\Omega}(\eps),
\end{equation}
where $|\Ct_\Omega(\eps)|= o(\eps^4)$ for $d=2$ and $|\Ct_\Omega(\eps)|= o(\eps^6)$ for $d=3$ uniformly in $\Omega$. 
\item[(b)] The expansion in (a) is equivalent to 
    \begin{equation}\label{eq:expansion_J3_ueps_u_}
        \int_{\Gamma_N} (u_{\Omega_\eps(x_1)} - u_\Omega)^2\;ds = -\frac{1}{2}c_{\Omega,x_1} q_{\Omega,x_1}(x_1)  + \Ct_{\Omega}(\eps).
\end{equation}
\end{itemize}
\end{lemma}
\begin{proof}
    In the following we set $u_\eps := u_{\Omega_\eps(x_1)}$ and $u:= u_\Omega$.

    \emph{ad (b):} This follows by testing the weak formulation of $q_{\Omega,x_1}$ with $d_{x_1}u_\Omega$ and testing the weak formulation of $d_{x_1}u_\Omega$ with $q_{\Omega,x_1}$:
    \begin{align}
    \int_{\Gamma_N} 2(d_{x_1}u_\Omega)^2\;ds& \stackrel{\eqref{eq:adjoint_fOmega_q_boundary}}{=} \int_{\Dsf} \nabla (q_{\Omega,x_1}) \cdot \nabla(d_{x_1}u_\Omega) + \gamma q_{\Omega,x_1} d_{x_1}u_\Omega \;dx 
    \stackrel{\eqref{eq:ueps_state_derivative_fomega_equation}}{=}  -c_{\Omega,x_1} q_{\Omega,x_1}(x_1). 
    \end{align}

    \emph{ad (a):} As in the proof of Lemma~\ref{lem:ueps_uOmega_J1_expansion} we write (see eqn. \eqref{eq:ueps_uomega_Deps}) for $d=2$
    \begin{align}\label{eq:ueps_u_Deps_boundary}
    \begin{split}
    u_\eps - u =&  \eps^2 |B|d_{x_1}u_\Omega + \eps^4 \frac12 |B|^2  \mathfrak d_{x_1}^2u \\
    & + \eps^4 \frac12 |B|^2 (\Uk{2}\circ T_\eps^{-1} - \Rki{2}{1}(x-x_1) + \Rki{2}{1}(\eps)) + \Ce_\Omega(\eps) \quad \text{ on } \Dsf\setminus \bar B_\eps(x_1)
\end{split}
\end{align}
and similarly for $d=3$:
    \begin{align}\label{eq:ueps_u_Deps_boundary_3d}
    \begin{split}
    u_\eps - u =&  \eps^3 |B|d_{x_1}u_\Omega + \eps^6 \frac12 |B|^2  \mathfrak d_{x_1}^2u \\
                & + \eps^6 \frac12 |B|^2 (\eps^{-1}\Uk{2}\circ T_\eps^{-1} - \Rki{2}{1}(x-x_1)) + \Ce_\Omega(\eps) \quad \text{ on } \Dsf\setminus \bar B_\eps(x_1).
\end{split}
\end{align}

Since $d_{x_1}u_\Omega = |B|^{-1}(\Rki{1}{1}(x-x_1) + \vk{1})$ the function $d_{x_1}u_\Omega$ is continuous on $\Gamma_N$. It only has a singularity at $x_1\in \Dsf\setminus\partial \Omega$. Therefore
\begin{equation}
    \int_{\Gamma_N}(d_{x_1}u)^2\;ds
\end{equation}
is well-defined. Expanding $(u_\eps-u)^2$ using \eqref{eq:ueps_u_Deps_boundary} we see that for $d=2$ (and similarly for  $d=3$ in view of \eqref{eq:ueps_u_Deps_boundary_3d})  it is sufficient to shows that the integrals of $(\mathfrak d_{x_1}^2u)^2$ and $(\Uk{2}\circ T_\eps^{-1} - \Rki{2}{1}(x-x_1) + \Rki{2}{1}(\eps))^2$ and $\Ce_\Omega(\eps)^2$ over $\Gamma_N$ are bounded. 
The other terms appearing in the expansion are products of $\Uk{2}\circ T_\eps^{-1} - \Rki{2}{1}(x-x_1) + \Rki{2}{1}(\eps), \Ce_\Omega(\eps), d_{x_1}u$ and $\mathfrak d_{x_1}^2u$ and can be estimated with H\"older's inequality. Clearly $\mathfrak d_{x_1}^2u$ is bounded on $\Gamma_N$ as it is continuous. Now since $|\Uk{2}(x)-\Rki{2}{1}(x)|\le C|x|^{-(d-1)}$ for all large $|x|$ and $d\in\{2,3\}$,  we have replacing $x$ with $T_\eps^{-1}(x)$ and using $\Rki{2}{1}(T_\eps^{-1}(x))=\Rki{2}{1}(x-x_1) - \Rki{2}{1}(\eps)$ for $d=2$ and $\Rki{2}{1}(T_\eps^{-1}(x)) = \eps \Rki{2}{1}(x-x_1)$  for $d=3$ that  
\begin{equation}
    |\eps^{2-d}\Uk{2}(T_\eps^{-1}(x))-\Rki{2}{1}(x-x_1) - c_d\Rki{2}{1}(\eps)| \le C \eps\frac{1}{|x-x_1|}, 
\end{equation}
where $c_d=1$ for $d=2$ and $c_d=0$ for $d=3$. Therefore
\begin{equation}
    \int_{\Gamma_N}(\eps^{d-2}\Uk{2}\circ T_\eps^{-1} -  \Rki{2}{1}(x-x_1) + c_d\Rki{2}{1}(\eps))^2\;ds \le C\eps^2 \int_{\Gamma_N}\frac{1}{|x-x_1|^2} \;ds. 
\end{equation}
Finally we estimate the integral over $(\Ce_\Omega(\eps))^2$. For this we use the trace theorem and H\"older's inequality to obtain
\begin{align}
    \begin{split}
    \int_{\Gamma_N}(\Ce_\Omega(\eps))^2 \;ds & \le \int_{\Dsf} |\Ce_\Omega(\eps)|+ |\nabla(\Ce_\Omega(\eps))| \;dx \\
                                             & = \int_{\Dsf} (\Ce_\Omega(\eps))^2 + 2 |\Ce_\Omega(\eps)\nabla(\Ce_\Omega(\eps))| \;dx\\
                                             & \le \int_{\Dsf} (\Ce_\Omega(\eps))^2 + 2 \|\Ce_\Omega(\eps)\|_{L^2(\Dsf)}\|\nabla(\Ce_\Omega(\eps))\|_{L^2(\Dsf)^d} \\
                                             & \le 2\|\Ce_\Omega(\eps)\|_{L^2(\Dsf)}^2 + \|\nabla(\Ce_\Omega(\eps))\|_{L^2(\Dsf)^d}^2\\
                                             & \le C\|\Ce_\Omega(\eps)\|_{H^1(\Dsf)}^2 \stackrel{Cor.~\ref{cor:expansion_remainder_D}}{=} \begin{cases}
                                                     O(\eps^{10}(-\ln(\eps))) & \text{ for } d=2\\
                                                     O(\eps^{14}) & \text{ for } d=3. 
                                             \end{cases}
\end{split} 
\end{align}
This finishes the proof.
\end{proof}

\subsubsection{Taylor-like expansion of $J_2(\Omega)$ with multiple holes}

For dimension $d=2$ we have the following result for multiple holes. We only formulate it in terms of the state derivatives $d_{x_i}u_\Omega$ and $d_{x_ix_j}^2u_\Omega$, however, using Theorem~\ref{thm:expansion_J2_fOmega} it is readily reformulated in terms of the adjoints $p_\Omega$ and $q_{\Omega,x_1}$. 
\begin{theorem}
Let $\Omega\subset \Dsf\subset \VR^2$ be an open set. For $x_1,\ldots,x_n\in \Dsf\setminus\partial \Omega$, $n\ge 1$ with $x_i\ne x_j$ for $i\ne j$, we have
\begin{align}
    J(\Omega_{\eps_1,\ldots,\eps_n})  = J(\Omega) + D\VJ (\Omega)(x_1,\ldots, x_n)\cdot \begin{pmatrix}
|B_{\eps_1}|\\ \vdots \\ |B_{\eps_n}|
\end{pmatrix} 
        + & \frac12D^2 \VJ(\Omega)(x_1,\ldots, x_n) \begin{pmatrix}
|B_{\eps_1}| \\\vdots \\ |B_{\eps_n}|
\end{pmatrix}\cdot \begin{pmatrix}
|B_{\eps_1}| \\ \vdots \\|B_{\eps_n}|
\end{pmatrix}  
                              \\
          &+ o(|B_{\eps_1}|^2+\cdots + |B_{\eps_n}|^2),
\end{align}
where for $i,j\in \{1,\ldots, n\}$:
\begin{align}
    (D^2 \VJ(\Omega)(x_1,\ldots, x_n))_{ii} & = \int_{\Gamma_N} 2 (d_{x_i}u_\Omega)^2\;ds + \int_{\Gamma_N} 2(u-u_r)\mathfrak d_{x_1}^2 u_\Omega \;ds \\
    (D^2 \VJ(\Omega)(x_1,\ldots, x_n))_{ij} & = \int_{\Gamma_N}  2(d_{x_i}u_\Omega) (d_{x_j} u_\Omega)\;ds, \quad i\ne j,
\end{align}
and for $i\in \{1,\ldots, n\}$:
\begin{equation}
    (D\VJ (\Omega)(x_1,\ldots, x_n))_i :=  \int_{\Gamma_N} 2(u-u_r)(d_{x_i}u_\Omega)\;ds.
\end{equation}
\end{theorem}
\begin{proof}
    We apply Theorem~\ref{main:thm_expansion_general} with $\ell(\eps):=0$, $d=2$ and $n\ge 1$. Assumption~(A1) is satisfied in view of Theorem~\ref{thm:expansion_J2_fOmega}. Assumption (A2): Similarly to the computation of \eqref{eq:diff_DJ(Omega_eps2)} we have for $x_1,x_2\in \Dsf\setminus\partial \Omega$,  $x_1\ne x_2$
        \begin{equation}
    DJ(\Omega_{\eps_2}(x_2))(x_1)-DJ(\Omega)(x_1) = -c_{\Omega,x_2} |B_{\eps_2}|q_{\Omega,x_1}(x_2) + \underbrace{(\eps_2)^4\int_0^1\int_B  \frac12 \nabla^2 q_{\Omega,x_1}(x_2 + sx )x\cdot x\;dx\,ds}_{=:\frR{2}_{\Omega,x_1,x_2}(\eps)},
\end{equation}
where $D^2J(\Omega)(x_1,x_2) = \int_\Dsf (d_{x_1}u_\Omega)(d_{x_2}u_\Omega)\;dx =  -c_{\Omega,x_1} q_{\Omega,x_1}(x_2)$ and $q_{\Omega,x_1}$ solves \eqref{eq:adjoint_fOmega_q_boundary}. This shows that Assumption~(A2) is satisfied. Assumption~(A3) is clearly satisfied since $\ell(\eps)=0$. To check Assumption~(A4) let $x_1,x_2,x_3\in \Dsf\setminus\partial \Omega$ with $x_1\ne x_2\ne x_3$. Then 
\begin{equation}
    D^2J(\Omega_{\eps_3})(x_1,x_2) = D^2J(\Omega)(x_1,x_2)
\end{equation}
since $d_{x_1}u_{\Omega_{\eps_3}} = d_{x_1}u_{\Omega}$, $i=1,2$, for $\eps_3>0$ small. For $x_1=x_2\ne x_3$ we have according to Theorem~\ref{thm:expansion_J2_fOmega} with $\Delta p_\Omega(x_1) = \gamma p_\Omega(x_1) + 2(u_\Omega-u_r)(x_1)$:
\begin{equation}
    D^2J(\Omega_{\eps_3})(x_1,x_1) = -c_{\Omega,x_1}\frac{1}{4\pi}(\gamma p_{\Omega_{\eps_3}}(x_1) + 2(u_{\Omega_{\eps_3}}-u_r)(x_1)).
\end{equation}
Then as in \eqref{eq:elliptic_regularity} using elliptic regularity near $x_3$, we obtain for $\delta>0$ small and $p>2$:
\begin{align}
    \begin{split}
    |p_{\Omega_{\eps_3}}(x_1) -  p_\Omega(x_1)| & \le \|p_{\Omega_{\eps_3}} -  p_\Omega\|_{W^{1,p}(B_{\delta}(x_3))}  
                                                 \le \|u_{\Omega_{\eps_3}} - u_\Omega\|_{L^p(B_{2\delta}(x_3))}\\
    |u_{\Omega_{\eps_3}}(x_1) -  u_\Omega(x_1)|  & \le \|u_{\Omega_{\eps_3}} -  u_\Omega\|_{W^{1,p}(B_{\delta}(x_3))} 
                                                  \le \|f_{\Omega_{\eps_3}}-f_\Omega\|_{L^p(B_{2\delta}(x_3))}. 
\end{split}
\end{align}
Therefore $|\frR{4}_{\Omega,x_1,x_2}(\eps_3)| = o(1)$ uniformly in $\Omega$ since:
\begin{align}
    \begin{split}
        |\frR{4}_{\Omega,x_1,x_2}(\eps_4)| & = | D^2J(\Omega_{\eps_3})(x_1,x_1)- D^2J(\Omega)(x_1,x_1)|  \\
                                           & \le C \|f_{\Omega_{\eps_3}}-f_\Omega\|_{L^p(B_{2\delta}(x_3))} \le C |B_{\eps_3}|^{\frac1p}.  
\end{split}
\end{align}
This finishes the proof.
\end{proof}

We finally state the result in dimension $d=3$ which follows again by an application of Theorem~\ref{main:thm_expansion_general} with $\ell(\eps):= \eps^{-1}|B_\eps|$. Since the proof is similar to the previous one we will omit it. 

\begin{theorem}
Let $\Omega\subset \Dsf$ be an open set. For $x_1,\ldots,x_n\in \Dsf\setminus\partial \Omega$, $n\ge 1$ with $x_i\ne x_j$ for $i\ne j$, we have
\begin{align}
    \begin{split}
    J(\Omega_{\eps_1,\ldots,\eps_n})  =  J(\Omega) + D\VJ (\Omega)(x_1,\ldots, x_n)\cdot \begin{pmatrix}
|B_{\eps_1}|\\ \vdots \\ |B_{\eps_n}|
\end{pmatrix} 
        &+ 
        \frac12D^{\frac32} \VJ(\Omega)(x_1,\ldots, x_n) \begin{pmatrix}
            \eps^{-\frac12}|B_{\eps_1}| \\\vdots \\ \eps^{-\frac12}|B_{\eps_n}|
\end{pmatrix}\cdot \begin{pmatrix}
\eps^{-\frac12}|B_{\eps_1}| \\ \vdots \\  \eps^{-\frac12}|B_{\eps_n}|
\end{pmatrix}  \\
          &+ \frac12D^2 \VJ(\Omega)(x_1,\ldots, x_n) \begin{pmatrix}
|B_{\eps_1}| \\\vdots \\ |B_{\eps_n}|
\end{pmatrix}\cdot \begin{pmatrix}
|B_{\eps_1}| \\ \vdots \\|B_{\eps_n}|
\end{pmatrix}     \\
          &+ o(|B_{\eps_1}|^2+\cdots + |B_{\eps_n}|^2),
    \end{split}
\end{align}
where for $i,j=1,\ldots, n$:
\begin{align}
    (D^{\frac32} \VJ(\Omega)(x_1,\ldots, x_n))_{ii} & = \int_{\Gamma_N} 2(u-u_r)\mathfrak d_{x_1}^2 u_\Omega \;ds \\
    (D^{\frac32} \VJ(\Omega)(x_1,\ldots, x_n))_{ij} & = 0, \quad i\ne j \\
    (D^2 \VJ(\Omega)(x_1,\ldots, x_n))_{ij} & = \int_{\Gamma_N}  2(d_{x_i}u_\Omega) (d_{x_j} u_\Omega)\;ds 
\end{align}
and for $i\in \{1,\ldots,n\}$:
\begin{equation}
    (D\VJ (\Omega)(x_1,\ldots, x_n))_i =  \int_{\Gamma_N} 2(u-u_r)(d_{x_i}u_\Omega)\;ds.
\end{equation}
\end{theorem}

\section*{Conclusion}
In this paper we studied second order topological expansions with several inclusions/holes for several cost functional $J(u_\Omega)$ subject to the PDE $-\Delta u_\Omega + \gamma u_\Omega = f_\Omega$ with mixed boundary conditions. 

We provided a general theorem that uses the expansion with respect to one inclusion/hole in order to obtain the more general case of multiple inclusions. 
We proved that the second order term has a structure arising as the iterated derivative of the cost functional using the notion of the topological state derivative. 

Although we treated the seemingly simple case of right hand side perturbations our methodology also applies to operator or lower order perturbations, e.g., PDEs of the form $-\Div(\beta_\Omega \nabla u_\Omega) + \alpha_\Omega u_\Omega =f$ with piecewise constant functions $\beta_\Omega$ and $\alpha_\Omega$ defined in the same manner as $f_\Omega$. The precise structure will evidently differ from our results therefore requires further research and is left for future investigations. 

Since we proved the quadratic nature of the topological expansion, at least in $d=2$, it seems possible that such an expansion can have numerical applications that should be examined in a future paper.

 \bibliographystyle{plain}
 \bibliography{multiple_holes.bib}
\end{document}